\documentclass[11pt]{article}
\pdfoutput=1

\usepackage[T1]{fontenc}
\usepackage[utf8]{inputenc}
\usepackage{lmodern}
\usepackage{amsmath,amssymb,amsthm,mathtools}
\usepackage{microtype}
\usepackage[a4paper,margin=30mm]{geometry}
\usepackage{float}
\usepackage{hyperref}
\hypersetup{hidelinks}

\newtheorem{theorem}{Theorem}[section]
\newtheorem{proposition}[theorem]{Proposition}
\newtheorem{lemma}[theorem]{Lemma}
\newtheorem{corollary}[theorem]{Corollary}
\theoremstyle{definition}
\newtheorem{definition}[theorem]{Definition}
\newtheorem{problem}[theorem]{Problem}
\theoremstyle{remark}

\newcommand{\R}{\mathbb{R}}
\newcommand{\dd}{\,\mathrm{d}}
\newcommand{\Aff}{\operatorname{Aff}}
\newcommand{\Span}{\operatorname{span}}
\newcommand{\qref}[1]{\eqref{#1}}

\title{\bfseries Continuous Families of Central Configurations\\
for Even-Power Homogeneous Potentials}

\author{Pawe\l{} Nurowski\\[2mm]
\small Center for Theoretical Physics, Polish Academy of Sciences\\
\small al. Lotnik\'ow 32/46, 02-668 Warsaw, Poland\\
\small \texttt{nurowski@cft.edu.pl}}

\date{}
\begin{document}
\begingroup
\centering
{\Large\bfseries Continuous Families of Central Configurations\\
for Even-Power Homogeneous Potentials\par}
\vspace{5pt}
{\normalsize Pawe\l{} Nurowski\par}
{\small Center for Theoretical Physics, Polish Academy of Sciences, Warsaw, Poland\par}
\vspace{7pt}
\endgroup

\begin{abstract}
  \small
We study central configurations for $n$ bodies interacting through pair
potentials
\[
        U_\sigma
        =
        \frac{\kappa}{\sigma}
        \sum_{i<j}m_i m_j r_{ij}^{\sigma},
\]
with logarithmic limit at $\sigma=0$.  For the Newtonian exponent
$\sigma=-1$, the classical Chazy--Wintner--Smale finiteness problem asks
whether, for prescribed positive masses, there are only finitely many
central configurations up to similarity.  Finiteness is known in several
important special cases and under additional dimensional or genericity
hypotheses.  It is therefore natural to ask whether an analogous finiteness
principle continues to hold for other homogeneous exponents.

We show that for every positive even exponent
\[
        \sigma=2k,
        \qquad k\ge2,
\]
the answer is emphatically negative.  In fact, these exponents admit an
abundance of continuous families of pairwise non-similar central
configurations.  The mechanism is provided by spherical designs: on a common
sphere the force kernel for $\sigma=2k$ restricts to a polynomial of degree
at most $k$, and consequently every positive weighted spherical $k$-design
is a central configuration.  We compute the corresponding central
multiplier explicitly.
\end{abstract}

\vspace{-0.5em}

\section{Introduction}\label{sec:introduction}
The Introduction is written so as to make the paper largely self-contained.
Subsections~\ref{subsec:basic-notions}--\ref{subsec:homogeneous-potentials}
recall the basic terminology, place the problem in the context of the
classical Newtonian finiteness question, and formulate its homogeneous
analogue.  Readers already familiar with this background may proceed
directly to Subsections~\ref{subsec:first-quartic-family}
and~\ref{subsec:second-sextic-family}, where the phenomenon studied in the
paper first appears in two explicit examples, and then to
Subsection~\ref{subsec:general-mechanism-main-results}, where the general
spherical-design mechanism and the main results are stated.

\subsection{Basic notions recalled}\label{subsec:basic-notions}

We begin by recalling the elementary terminology from the $n$-body problem
that will be used throughout the paper.  Let
\[
        m_1,\ldots,m_n>0
\]
be fixed masses, and let
\[
        q=(q_1,\ldots,q_n),
        \qquad q_i\in\R^D,
\]
be a configuration of $n$ pairwise distinct points.  We write
\[
        r_{ij}=|q_i-q_j|,
        \qquad
        M=\sum_{i=1}^n m_i,
\]
and denote by
\begin{equation}\label{eq:center-of-mass-intro}
        c=\frac{1}{M}\sum_{i=1}^n m_iq_i
\end{equation}
the center of mass of the configuration.

For the Newtonian problem, corresponding to the homogeneous exponent
$\sigma=-1$, the potential is
\begin{equation}\label{eq:newton-potential}
        U_{-1}(q_1,\ldots,q_n)
        =-\kappa\sum_{1\le i<j\le n}\frac{m_i m_j}{r_{ij}},
        \qquad \kappa>0.
\end{equation}
If $x_i=x_i(t)$ are the trajectories of the bodies, Newton's equations are
\begin{equation}\label{eq:newton-equations-intro}
        \ddot x_i
        =\kappa\sum_{j\ne i}
        m_j\frac{x_j-x_i}{|x_j-x_i|^3},
        \qquad i=1,\ldots,n.
\end{equation}
For a fixed configuration $q$ we denote the corresponding Newtonian
acceleration field by
\begin{equation}\label{eq:newton-field-intro}
        F_i^{(-1)}(q)
        =\sum_{j\ne i}
        m_j\frac{q_j-q_i}{r_{ij}^3}.
\end{equation}

Recall that $q$ is called a \emph{central configuration} if there exists a
real number $\lambda$, the same for all bodies, such that
\begin{equation}\label{eq:newton-central-intro}
        F_i^{(-1)}(q)
        =\lambda(q_i-c),
        \qquad i=1,\ldots,n.
\end{equation}
Thus the acceleration field at every body is directed along the line through
that body and the center of mass, and its magnitude is proportional to the
distance from the center of mass, with one common proportionality factor for
the whole configuration.  In center-of-mass coordinates, where $c=0$,
\qref{eq:newton-central-intro} becomes simply
\[
        F_i^{(-1)}(q)=\lambda q_i.
\]
Central configurations play a distinguished dynamical role: they are the
fixed shapes underlying homothetic motions and, in suitable even-dimensional
situations, rotating homographic motions.  The precise separation between
the time-dependent motion and the time-independent central-configuration
equations will be derived in Section~2.

We also recall the elementary affine-geometric notion that will be used to
describe the dimension of a configuration.  The affine hull of the points
$q_1,\ldots,q_n$ is
\begin{equation}\label{eq:affine-hull-definition}
        \Aff\{q_1,\ldots,q_n\}
        =
        q_1+
        \Span\{q_2-q_1,\ldots,q_n-q_1\}.
\end{equation}
Thus the affine hull is the smallest affine subspace containing all the
points of the configuration.  Since the linear space on the right-hand side
of \qref{eq:affine-hull-definition} is generated by only $n-1$ vectors, its
dimension cannot exceed $n-1$.

The \emph{dimension of the configuration}
$q=(q_1,\ldots,q_n)$ is, by definition,
\begin{equation}\label{eq:configuration-dimension}
        \dim q
        :=
        \dim\Aff\{q_1,\ldots,q_n\}.
\end{equation}
In particular,
\begin{equation}\label{eq:configuration-dimension-bound}
        \dim q\le \min\{D,n-1\}.
\end{equation}
The center of mass $c$ belongs to the affine hull, since it is an affine
combination of the points $q_i$.  After translating $c$ to the origin, the
affine hull therefore becomes the linear space
\begin{equation}\label{eq:centered-configuration-span}
        \Aff\{q_1,\ldots,q_n\}-c
        =
        \Span\{q_1-c,\ldots,q_n-c\},
\end{equation}
and hence
\begin{equation}\label{eq:centered-configuration-dimension}
        \dim q
        =
        \dim\Span\{q_1-c,\ldots,q_n-c\}.
\end{equation}

Thus a collinear configuration has dimension $1$, a non-collinear planar
configuration has dimension $2$, and a non-planar configuration in
three-dimensional space has dimension $3$.  Whenever we refer to a
$d$-dimensional configuration, the symbol $d$ means $\dim q$; the ambient
Euclidean space has dimension $D$.

Finally, the finiteness questions considered in this paper concern shapes,
not particular placements of the same shape in Euclidean space.  For fixed
labelled masses we call two configurations similar if
\begin{equation}\label{eq:similarity-equivalence}
        q_i'=a+sQq_i,
        \qquad i=1,\ldots,n,
\end{equation}
for some translation $a\in\R^D$, some scale $s>0$, and some orthogonal
transformation $Q\in O(D)$.  When some masses are equal we shall also allow
the corresponding finite permutations when discussing unlabelled geometric
families.  Such finite identifications do not affect the local dimensions of
the moduli spaces considered below.

\subsection{The classical finiteness problem}\label{subsec:classical-finiteness}

With this terminology in place, the classical Newtonian finiteness problem
can be stated as follows: for a prescribed vector of positive masses
\[
        (m_1,\ldots,m_n),
\]
are there only finitely many similarity classes of Newtonian central
configurations?

The question is traditionally traced back to Chazy's 1918 study of the
$n$-body problem \cite{Chazy1918}, appears explicitly in Wintner's monograph
\cite{Wintner1941}, and was later singled out by Smale as Problem~6 in his
list of mathematical problems for the twenty-first century
\cite{Smale1998}.  In its familiar planar form it is commonly referred to
as the Chazy--Wintner--Smale finiteness problem.  Important later results
include Moeckel's \emph{generic} finiteness theorem for configurations of
affine dimension $n-2$, usually called Dziobek configurations,
\cite{Moeckel2001}, the proof by Hampton and Moeckel of finiteness for the
Newtonian four-body problem \cite{HamptonMoeckel2006}, and the result of
Albouy and Kaloshin for five planar bodies, which proves finiteness outside
a specified codimension-two exceptional set in the positive mass space
\cite{AlbouyKaloshin2012}.  For larger numbers of bodies the general
Newtonian finiteness problem remains open.

In Moeckel's result, the qualifier \emph{generic} is essential.  His theorem
proves finiteness for mass vectors belonging to a generic subset of the
positive mass space; it does not assert finiteness for every prescribed
positive mass vector.  Mass vectors outside that generic subset are therefore
not covered by the theorem.  This is logically different from the full
classical finiteness problem stated above, which asks for finiteness for each
fixed choice of positive masses.

\subsection{Homogeneous potentials}\label{subsec:homogeneous-potentials}

The Newtonian potential belongs to the one-parameter family of homogeneous
pair potentials
\begin{equation}\label{eq:homogeneous-potential}
        U_{\sigma}(q_1,\ldots,q_n)
        =\frac{\kappa}{\sigma}
        \sum_{1\le i<j\le n}m_i m_j r_{ij}^{\sigma},
        \qquad \sigma\ne0.
\end{equation}
After subtraction of an irrelevant $\sigma$-dependent additive constant, the
limit $\sigma\to0$ is the logarithmic potential
\begin{equation}\label{eq:log-potential}
        U_0(q_1,\ldots,q_n)
        =\kappa\sum_{1\le i<j\le n}m_i m_j\log r_{ij}.
\end{equation}
The acceleration field associated with \qref{eq:homogeneous-potential} is
\begin{equation}\label{eq:homogeneous-field-intro}
        F_i^{(\sigma)}(q)
        =\sum_{j\ne i}
        m_j(q_j-q_i)r_{ij}^{\sigma-2},
\end{equation}
and the equations of motion are
\begin{equation}\label{eq:homogeneous-equations-intro}
        \ddot x_i=\kappa F_i^{(\sigma)}(x).
\end{equation}
Accordingly, a configuration $q$ is central for the exponent $\sigma$ if
there exists a single real number $\lambda$ such that
\begin{equation}\label{eq:homogeneous-central-intro}
        F_i^{(\sigma)}(q)
        =\lambda(q_i-c),
        \qquad i=1,\ldots,n.
\end{equation}
The Newtonian problem is recovered at $\sigma=-1$.  Other distinguished
values are the logarithmic limit $\sigma=0$, the Hooke exponent $\sigma=2$,
and the positive even exponents
\[
        \sigma=4,6,8,\ldots .
\]

The Newtonian exponent $\sigma=-1$ is only one member of the family
\qref{eq:homogeneous-potential}.  It is therefore natural to ask whether the
finiteness phenomena discussed in the Newtonian regime persist when the
homogeneity exponent is changed.  The generic version of this question will
be our starting point.

\begin{problem}[Generic finiteness for homogeneous potentials]
\label{prob:generic-homogeneous-finiteness}
Fix the number of bodies $n$, an integer $d$ with $1\le d\le n-1$, and a
homogeneity exponent $\sigma$.  For a generic choice of positive masses
$(m_1,\ldots,m_n)$, are there only finitely many similarity classes of
central configurations $q$ satisfying
\[
        \dim q=d
\]
for the potential \qref{eq:homogeneous-potential}, with the logarithmic
interpretation \qref{eq:log-potential} when $\sigma=0$?
\end{problem}

For $\sigma=-1$, Problem~\ref{prob:generic-homogeneous-finiteness} reduces to
the generic Newtonian finiteness problem.  Our interest here is in
non-Newtonian exponents.  One value has to be separated immediately.  When
$\sigma=2$, equation \qref{eq:homogeneous-field-intro} gives
\begin{equation}\label{eq:hooke-field-intro}
\begin{aligned}
        F_i^{(2)}(q)
        &=\sum_{j\ne i}m_j(q_j-q_i) \\
        &=\sum_{j=1}^n m_jq_j-Mq_i \\
        &=M(c-q_i).
\end{aligned}
\end{equation}
Thus every configuration is central for $\sigma=2$, with $\lambda=-M$.
The Hooke exponent is therefore globally degenerate and must be excluded
from any nontrivial finiteness statement.

The exponents of primary interest in this paper are the positive even values
\begin{equation}\label{eq:positive-even-exponents}
        \sigma=2k,\qquad k\ge2.
\end{equation}
We do not restrict the dimension of the configurations to the plane.  The
structural results are formulated for arbitrary $d$.  Planar configurations
play a special role only because they provide the simplest explicit families
and admit a particularly transparent description in terms of regular
polygons and Fourier modes.

There is an important partial generic-finiteness result for homogeneous
potentials in the class of Dziobek configurations.  Dias
\cite{Dias2017} considers the family
\begin{equation}\label{eq:dias-potential}
        U_a
        =\frac{1}{2a+2}
        \sum_{i<j}m_i m_j r_{ij}^{2a+2},
        \qquad a\in\tfrac12\mathbb Z,\quad a\ne-1,
\end{equation}
with the logarithmic potential at $a=-1$.  In the notation of
\qref{eq:homogeneous-potential}, the two parameters are related by
\begin{equation}\label{eq:sigma-a-conversion}
        \sigma=2a+2.
\end{equation}
Hence the positive even exponents \qref{eq:positive-even-exponents}
correspond to
\[
        a=k-1\in\mathbb N
\]
and lie within the exponent range considered by Dias.

The dimensional hypothesis in Dias's theorem is, however, very specific:
it concerns configurations satisfying
\begin{equation}\label{eq:dziobek-dimension}
        \dim q=n-2.
\end{equation}
Thus his theorem addresses the single Dziobek slice $d=n-2$ of
Problem~\ref{prob:generic-homogeneous-finiteness}.  For example, in the
planar case $d=2$ this condition is satisfied only for $n=4$; planar
configurations with $n\ge5$ are outside its scope.  More generally, for a
fixed number of bodies $n$, the theorem does not address configuration
dimensions $d\ne n-2$.  The positive even exponents therefore provide a
natural setting in which to investigate the generic finiteness problem
beyond the Dziobek dimensional class.

\subsection{A first family at the quartic exponent}
\label{subsec:first-quartic-family}

We now give the simplest example illustrating what can happen at a positive
even exponent.  Identify the Euclidean plane with $\mathbb C$ and put
\begin{equation}\label{eq:omega-definition}
        \omega=e^{2\pi i/3}.
\end{equation}
Let
\begin{equation}\label{eq:triangle-definition}
        P_3=\{1,\omega,\omega^2\}
\end{equation}
be the equilateral triangle on the unit circle.  For a real parameter $\phi$
define the six-point configuration
\begin{equation}\label{eq:quartic-two-triangles}
        X_\phi=P_3\cup e^{i\phi}P_3.
\end{equation}
All six masses are taken to be equal to one.  The two triangles meet precisely
when
\begin{equation}\label{eq:quartic-collision-condition}
        \phi\equiv0\pmod{2\pi/3},
\end{equation}
so away from these values $X_\phi$ consists of six distinct points on the
unit circle.

Consider now the homogeneous potential with exponent $\sigma=4$.  In the
notation of \qref{eq:homogeneous-field-intro}, its force field is
\begin{equation}\label{eq:quartic-field}
        F_i^{(4)}(q)
        =\sum_{j\ne i}m_j(q_j-q_i)|q_j-q_i|^2.
\end{equation}
For the configuration \qref{eq:quartic-two-triangles}, a direct calculation
gives
\begin{equation}\label{eq:quartic-two-triangles-centrality}
        F_i^{(4)}(X_\phi)=-18q_i,
        \qquad i=1,\ldots,6,
\end{equation}
for every collision-free value of $\phi$.  Thus every such $X_\phi$ is a
central configuration for the quartic potential, with the same central
multiplier
\begin{equation}\label{eq:quartic-two-triangles-lambda}
        \lambda=-18.
\end{equation}
The remarkable point is that the relative angle $\phi$ remains free.

These configurations are strictly convex.  Indeed, all six points lie on one
circle and, away from the collision values
\qref{eq:quartic-collision-condition}, they are pairwise distinct; hence each
of them is a vertex of the convex hull.  The parameter $\phi$ is not merely
produced by rotating the whole configuration.  To see this, write the six
points as complex numbers $z_1,\ldots,z_6$ and consider
\begin{equation}\label{eq:quartic-shape-invariant}
        J_3(X_\phi)
        :=\frac{1}{6}\left|\sum_{j=1}^6 z_j^3\right|.
\end{equation}
For \qref{eq:quartic-two-triangles}, one obtains
\begin{equation}\label{eq:quartic-shape-invariant-value}
        J_3(X_\phi)
        =\left|\cos\frac{3\phi}{2}\right|.
\end{equation}
To use \qref{eq:quartic-shape-invariant} to distinguish similarity classes,
we first eliminate the translational and scaling parts of a possible
similarity.  Each of the two equilateral triangles in
\qref{eq:quartic-two-triangles} is centered at the origin, and therefore
\[
        \frac16\sum_{z\in X_\phi}z=0
\]
for every $\phi$.  Hence every member of the family has center of mass equal
to the origin.  If a similarity
\[
        z\longmapsto a+sQz,
        \qquad s>0,\quad Q\in O(2),
\]
maps one member of the family to another, then it maps the center of mass of
the first configuration to the center of mass of the second.  Since both
centers are zero, necessarily
\[
        a=0.
\]
All points of both configurations lie on the unit circle, so after
$a=0$ we also have
\[
        1=|sQz|=s
\]
for every vertex, and hence $s=1$.

Thus any similarity between two members of the normalized family reduces to
an orthogonal transformation, possibly followed by a permutation of the six
equal masses.  Under a rotation the sum
\[
        \sum z_j^3
\]
is multiplied by a complex number of modulus one. Under a reflection the
sum is complex conjugated up to a unit phase, so its absolute value is
unchanged. Consequently $J_3$ in \qref{eq:quartic-shape-invariant} is indeed an
invariant of the similarity class within this family.  On the interval
\begin{equation}\label{eq:quartic-fundamental-interval}
        0<\phi<\frac{\pi}{3},
\end{equation}
the function in \qref{eq:quartic-shape-invariant-value} is strictly
decreasing.  Consequently the family \qref{eq:quartic-two-triangles}
contains a continuum of pairwise non-similar, strictly convex central
configurations for one fixed positive mass vector and one fixed exponent
$\sigma=4$.

\subsection{A second family at the sextic exponent}
\label{subsec:second-sextic-family}

The preceding phenomenon is not peculiar to the quartic potential.  A second
example already appears at the next positive even exponent.  Put
\begin{equation}\label{eq:square-definition}
        P_4=\{1,i,-1,-i\}
\end{equation}
and, for a real parameter $\phi$, consider the eight-point configuration
\begin{equation}\label{eq:sextic-two-squares}
        Y_\phi=P_4\cup e^{i\phi}P_4.
\end{equation}
Again all masses are equal to one.  The two squares meet precisely when
\begin{equation}\label{eq:sextic-collision-condition}
        \phi\equiv0\pmod{\pi/2},
\end{equation}
so otherwise $Y_\phi$ consists of eight distinct points on the unit circle.

For the exponent $\sigma=6$, the force field
\qref{eq:homogeneous-field-intro} becomes
\begin{equation}\label{eq:sextic-field}
        F_i^{(6)}(q)
        =\sum_{j\ne i}m_j(q_j-q_i)|q_j-q_i|^4.
\end{equation}
A direct calculation for \qref{eq:sextic-two-squares} gives
\begin{equation}\label{eq:sextic-two-squares-centrality}
        F_i^{(6)}(Y_\phi)=-80q_i,
        \qquad i=1,\ldots,8,
\end{equation}
for every collision-free value of $\phi$.  Thus the relative angle is again
free and the central multiplier is
\begin{equation}\label{eq:sextic-two-squares-lambda}
        \lambda=-80.
\end{equation}

As in the quartic example, every collision-free member of the family is
strictly convex.  To distinguish its similarity classes, write its points as
$z_1,\ldots,z_8\in\mathbb C$ and define
\begin{equation}\label{eq:sextic-shape-invariant}
        J_4(Y_\phi)
        :=\frac{1}{8}\left|\sum_{j=1}^8 z_j^4\right|.
\end{equation}
Since each fourth power is equal to $1$ on the first square and to
$e^{4i\phi}$ on the second, one obtains
\begin{equation}\label{eq:sextic-shape-invariant-value}
        J_4(Y_\phi)=|\cos 2\phi|.
\end{equation}
Exactly as in the quartic family, every configuration $Y_\phi$ has center of
mass equal to the origin and all its vertices lie on the unit circle.
Therefore any similarity between two members of the normalized family has
zero translation and scale factor one, and hence reduces to an orthogonal
transformation, possibly followed by a permutation of the eight equal
masses.  Under a rotation the sum
\[
        \sum z_j^4
\]
is multiplied by a complex number of modulus one. Under a reflection the
sum is complex conjugated up to a unit phase, so its absolute value is
unchanged.  Thus $J_4$ in \qref{eq:sextic-shape-invariant} is an invariant of the similarity
class within this family. On the interval
\begin{equation}\label{eq:sextic-fundamental-interval}
        0<\phi<\frac{\pi}{4},
\end{equation}
it is strictly decreasing.  Hence \qref{eq:sextic-two-squares} contains a
continuum of pairwise non-similar, strictly convex central configurations for
one fixed positive mass vector and the fixed exponent $\sigma=6$.

\subsection{The general mechanism and the main results}
\label{subsec:general-mechanism-main-results}

The quartic and sextic families described above suggest that the continuous
families are not a consequence of the particular geometry of triangles and
squares alone.  What the two examples have in common is a more general
averaging property of their vertices on the circle.  In order to formulate
this property, and subsequently the mechanism behind both examples, we first
recall the notion of a weighted spherical design.

Let $V\subset\mathbb R^D$ be a $d$-dimensional Euclidean linear subspace, and
let
\begin{equation}\label{eq:sphere-radius-R}
        S_R(V):=\{x\in V:|x|=R\}
\end{equation}
be the sphere of radius $R$ in $V$.  Thus $S_R(V)$ has dimension $d-1$,
independently of the dimension $D$ of the ambient Euclidean space.  We denote
by $d\nu_R$ the normalized rotation-invariant measure on $S_R(V)$, so that
\begin{equation}\label{eq:normalized-spherical-measure}
        \int_{S_R(V)}d\nu_R=1.
\end{equation}
Let $q_1,\ldots,q_n\in S_R(V)$ carry positive weights
$m_1,\ldots,m_n$, and put
\[
        M=\sum_{i=1}^n m_i.
\]

The following definition expresses when averaging a polynomial over finitely
many weighted points gives exactly the same result as averaging that
polynomial uniformly over the whole sphere.

\begin{definition}[Weighted spherical $k$-design]
\label{def:weighted-spherical-k-design}
Let $k\ge1$ be an integer.  The weighted configuration
\[
        (q_1,\ldots,q_n;m_1,\ldots,m_n)
\]
is called a \emph{weighted spherical $k$-design} if
\begin{equation}\label{eq:weighted-k-design-definition}
        \frac{1}{M}\sum_{i=1}^n m_i p(q_i)
        =
        \int_{S_R(V)}p(x)\,d\nu_R(x)
\end{equation}
for every real polynomial $p:V\to\mathbb R$ of total degree at most $k$.

Thus $k$ specifies the degree up to which the finite weighted configuration
reproduces exactly the corresponding spherical averages.  In particular,
every spherical $k$-design is automatically a spherical $\ell$-design for
every $\ell\le k$.  The largest integer $k$ for which
\qref{eq:weighted-k-design-definition} holds is called the \emph{strength}
of the design.
\end{definition}

The two examples above can now be described in the same language.  The
equal-weight regular triangle $P_3$ is a spherical $2$-design, whereas the
equal-weight square $P_4$ is a spherical $3$-design.  More generally, the
equal-weight regular $(k+1)$-gon $P_{k+1}$ is a spherical $k$-design on the
circle.

At the same time, the homogeneity exponents in the first two examples are
\[
        4=2\cdot2,
        \qquad
        6=2\cdot3.
\]
This suggests the general correspondence
\begin{equation}\label{eq:exponent-design-correspondence}
        \sigma=2k
        \qquad\longleftrightarrow\qquad
        \text{spherical design of strength at least $k$}.
\end{equation}
Thus $\sigma=4$ requires strength $2$, $\sigma=6$ requires strength $3$,
$\sigma=8$ requires strength $4$, and so on.  The next theorem shows that
this correspondence is not merely a pattern exhibited by the two examples:
it gives the general mechanism producing central configurations for all
positive even homogeneity exponents.

\begin{theorem}[Even-power design theorem]
\label{thm:even-power-design}
Let $V\subset\mathbb R^D$ be a $d$-dimensional Euclidean linear subspace,
with $d\ge2$, and let $k\ge1$.  Suppose that positive masses
$m_1,\ldots,m_n$ at points
\[
        q_1,\ldots,q_n\in S_R(V)
\]
form a weighted spherical $k$-design.  Then the configuration is central
for the homogeneous exponent $\sigma=2k$.  More precisely,
\begin{equation}\label{eq:general-even-centrality}
        \sum_{j\ne i}m_j(q_j-q_i)|q_j-q_i|^{2k-2}
        =\lambda_{d,k}q_i,
        \qquad i=1,\ldots,n,
\end{equation}
where
\begin{equation}\label{eq:general-even-lambda}
        \lambda_{d,k}
        =
        -2^{2k-1}MR^{2k-2}
        \frac{\bigl((d-1)/2\bigr)_k}{(d-1)_k}.
\end{equation}
Here
\begin{equation}\label{eq:pochhammer-definition}
        (a)_k=a(a+1)\cdots(a+k-1)
\end{equation}
denotes the rising factorial.
\end{theorem}

For $d=2$, \qref{eq:general-even-lambda} takes the particularly simple form
\begin{equation}\label{eq:planar-general-even-lambda}
        \lambda_{2,k}
        =-\frac12\binom{2k}{k}MR^{2k-2}.
\end{equation}
The values $-18$ in \qref{eq:quartic-two-triangles-lambda} and $-80$ in
\qref{eq:sextic-two-squares-lambda} are the cases $k=2$ and $k=3$ of
\qref{eq:planar-general-even-lambda}, with respectively $M=6$ and $M=8$ and
$R=1$.

A second feature of the two examples is the freedom to rotate one block
relative to the other.  The design-theoretic part of this phenomenon is
classical.  We shall use the following standard closure property of weighted
spherical designs.  In the language of cubature rules it is simply the fact
that positive combinations of exact normalized quadrature rules remain
exact after renormalization.  An explicit weighted formulation is given, for
example, in \cite{MohammadpourWaldron2019}.

\begin{lemma}[Standard closure of weighted spherical designs]
\label{lem:weighted-design-closure}
Let $V\subset\mathbb R^D$ be a $d$-dimensional Euclidean linear subspace, and
let $X_1,\ldots,X_s$ be weighted spherical $k$-designs on the same sphere
$S_R(V)$.  For each block choose an arbitrary constant $c_a>0$ and an
orthogonal transformation $Q_a\in O(V)$.  Multiply every weight in $X_a$ by
$c_a$, apply $Q_a$ to all its points, and then take the weighted union of the
resulting blocks.  The union is again a weighted spherical $k$-design on
$S_R(V)$.
\end{lemma}

The constants $c_a$ are independent.  Thus the total masses carried by the
different design blocks may be chosen independently and need not be equal.
For an equal-weight block, such as a regular polygon or the vertex set of a
Platonic solid, this means that all vertices of a given block have one common
positive mass, but this common mass may be different from block to block.

Combining this classical closure property with
Theorem~\ref{thm:even-power-design} gives the mechanical consequence that we
shall use repeatedly.

\begin{corollary}[Relative rotations of central configurations from designs]
\label{cor:relative-rotations-designs}
Under the assumptions of Lemma~\ref{lem:weighted-design-closure}, for
$\sigma=2k$ every collision-free configuration in the family
\begin{equation}\label{eq:rotated-design-union}
        Q_1X_1\cup\cdots\cup Q_sX_s
\end{equation}
is central.  The positive total masses of the individual blocks are
independent parameters.
\end{corollary}

Corollary~\ref{cor:relative-rotations-designs} proves centrality, but by
itself it does \emph{not} yet prove that the rotation parameters give
different similarity classes.  This distinction is important.  The group
$SO(d)$ has dimension $d(d-1)/2$, but a parameter count alone does not show
that the map from relative orientations to geometric configurations is
locally finite-to-one.  One must prove that, for generic relative
orientations, the individual rotated blocks can be recovered from the metric
of their union using the internal distances of the fixed blocks, and that
only finitely many orientation choices can produce the same weighted
configuration up to similarity.

The following theorem gives precisely this missing quotient statement.

\begin{theorem}[Generic moduli of independently rotated design blocks]
\label{thm:generic-rotated-design-moduli}
Let $k\ge2$, let $V$ be a $d$-dimensional Euclidean space with $d\ge2$, and
let
\[
        X_1,\ldots,X_s\subset S_R(V),
        \qquad s\ge2,
\]
be finite weighted spherical $k$-designs with fixed positive weights.
For
\[
        (Q_1,\ldots,Q_s)\in SO(V)^s
\]
consider the weighted union
\begin{equation}\label{eq:general-rotated-block-family}
        \mathcal X(Q_1,\ldots,Q_s)
        :=
        Q_1X_1\cup\cdots\cup Q_sX_s.
\end{equation}
Then there exists an open dense subset
\[
        \mathcal U\subset SO(V)^s
\]
such that every configuration \qref{eq:general-rotated-block-family} with
$(Q_1,\ldots,Q_s)\in\mathcal U$ is collision-free and has the following
properties.

\begin{enumerate}
\item[(i)]
Once the fixed block data $X_1,\ldots,X_s$ are given, the decomposition into
the $s$ rotated design blocks can be recovered from the Euclidean metric of
the weighted union: one uses the finite set of distances which occur inside
the fixed blocks.  In particular, every isometry between two generic unions
must permute the rotated blocks, up to a finite permutation of mutually
congruent blocks carrying compatible weights.

\item[(ii)]
After quotienting by a common orthogonal transformation of $V$, by the
finite orthogonal symmetry groups of the individual blocks, and by the
finite permutations allowed in {\rm(i)}, the map from relative orientations
to similarity classes is locally finite-to-one.

\item[(iii)]
Consequently the generic local dimension of the resulting moduli space of
weighted configurations is
\begin{equation}\label{eq:general-rotated-block-moduli-dimension}
        (s-1)\dim SO(d)
        =
        (s-1)\frac{d(d-1)}2.
\end{equation}
\end{enumerate}

By Corollary~\ref{cor:relative-rotations-designs}, all these collision-free
configurations are central for $\sigma=2k$.
\end{theorem}

For two blocks, Theorem~\ref{thm:generic-rotated-design-moduli} gives the
particularly important dimension
\begin{equation}\label{eq:two-block-moduli-dimension}
        \frac{d(d-1)}2.
\end{equation}
Thus two independently oriented design blocks give generically a
one-dimensional family in configuration dimension $d=2$, a
three-dimensional family in configuration dimension $d=3$, a
six-dimensional family in configuration dimension $d=4$, and so on.
The integers $d$ and $k$ play completely different roles: $d$ controls the
dimension of the rotation group, whereas $k$ controls the positive even
homogeneity exponent $\sigma=2k$ through
Theorem~\ref{thm:even-power-design}.

We next specialize these general results to the most elementary symmetric
design blocks.

A single regular polygon with equal masses at its vertices is central for
every homogeneous exponent $\sigma$ for which the force is defined.  The
same is true for each of the five Platonic solids with equal masses at its
vertices.  We briefly justify this symmetry argument.

Let $G\subset O(V)$ be the full orthogonal symmetry group of the regular
polygon or Platonic solid.  Since all vertex masses are equal, every
$g\in G$ acts by a permutation of the weighted vertices.  The force law
\qref{eq:homogeneous-field-intro} is orthogonally equivariant: if $gq_i=q_j$,
then
\begin{equation}\label{eq:symmetric-block-force-equivariance}
        F_j^{(\sigma)}(q)
        =
        gF_i^{(\sigma)}(q).
\end{equation}
Indeed, $g$ preserves all pairwise distances and carries every difference
$q_\ell-q_i$ to the corresponding difference between the permuted
vertices.

Fix a vertex $q_i$.  Every element of the stabilizer
\[
        G_{q_i}:=\{g\in G:gq_i=q_i\}
\]
therefore fixes the vector $F_i^{(\sigma)}(q)$.  For a regular polygon, using
the full dihedral symmetry group, the common fixed subspace of $G_{q_i}$ is
the radial line $\mathbb R q_i$; the same is true for the stabilizer of a
vertex of each Platonic solid.  Hence
\[
        F_i^{(\sigma)}(q)=\lambda_i q_i
\]
for some scalar $\lambda_i$.

Finally, $G$ acts transitively on the vertices.  If $gq_i=q_j$, then
\qref{eq:symmetric-block-force-equivariance} gives
\[
        F_j^{(\sigma)}(q)
        =
        gF_i^{(\sigma)}(q)
        =
        \lambda_i gq_i
        =
        \lambda_i q_j.
\]
Thus $\lambda_j=\lambda_i$, so the radial coefficient is the same at every
vertex.  Therefore the configuration is central.  In particular, the
centrality of one regular polygon or one Platonic solid does not depend on a
special value of $\sigma$.

The role of the spherical-design strength is different.  It determines for
which positive even exponents the centrality survives after several such
blocks, placed on the same sphere, are rotated independently of one another.
The relevant design strengths are classical; see
\cite{DelsarteGoethalsSeidel1977,BannaiBannai2009}.  Throughout the
polygonal and Platonic constructions below, ``equal masses on a block''
means equal masses only within that block.  Different blocks may carry
different common vertex masses.

For regular polygons the design strengths are summarized in
Table~\ref{tab:polygon-design-strengths}.

\begin{table}[H]
\centering
\begin{tabular}{c|c|c}
regular block & design strength &
positive even exponents allowed by Corollary~\ref{cor:relative-rotations-designs}
\\
\hline
$P_3$ & $2$ & $\sigma=4$ \\
$P_4$ & $3$ & $\sigma=4,6$ \\
$P_5$ & $4$ & $\sigma=4,6,8$ \\
$P_n$ & $n-1$ & $\sigma=2k$, $2\le k\le n-1$
\end{tabular}
\caption{Spherical-design strengths of equal-weight regular polygons on the
circle and the corresponding positive even exponents for arbitrary relative
rotations of design blocks.}
\label{tab:polygon-design-strengths}
\end{table}

In three dimensions the analogous data for the Platonic solids are shown in
Table~\ref{tab:platonic-design-strengths}.  The tetrahedron has strength $2$,
the cube and the octahedron have strength $3$, and the dodecahedron and the
icosahedron have strength $5$ \cite{BannaiBannai2009}.  Notice again that
these strengths do not describe the exponents for which a \emph{single}
Platonic solid is central: each single solid is central for every $\sigma$.
They describe instead the exponents for which arbitrary relative rotations of
several Platonic design blocks remain central by
Corollary~\ref{cor:relative-rotations-designs}.

\begin{table}[H]
\centering
\begin{tabular}{c|c|c|c}
Platonic solid & vertices & design strength & positive even exponents allowed \\
\hline
tetrahedron & $4$ & $2$ & $\sigma=4$ \\
cube & $8$ & $3$ & $\sigma=4,6$ \\
octahedron & $6$ & $3$ & $\sigma=4,6$ \\
dodecahedron & $20$ & $5$ & $\sigma=4,6,8,10$ \\
icosahedron & $12$ & $5$ & $\sigma=4,6,8,10$
\end{tabular}
\caption{Spherical-design strengths of the equal-weight vertex sets of the
five Platonic solids and the positive even exponents covered by
Corollary~\ref{cor:relative-rotations-designs}.}
\label{tab:platonic-design-strengths}
\end{table}

These observations give continuous families in dimensions two and three
without requiring the individual design blocks to be congruent.  What
matters is not that the blocks have the same shape, or even the same design
strength, but that they are spherical $k$-designs on the same sphere for the
value of $k$ under consideration.  Equivalently, each block must have design
strength at least $k$.

\begin{corollary}[Regular polygonal and Platonic relative-rotation families]
\label{cor:polygonal-platonic-relative-rotation-families}
Let all blocks in each of the following statements be inscribed in the same
sphere.

\begin{enumerate}
\item[(i)]
Let $P_m$ and $P_n$ be regular polygons on the same circle, where
\[
        m\ge k+1,
        \qquad
        n\ge k+1.
\]
Give every vertex of $P_m$ one common mass $\mu_m>0$, and every vertex of
$P_n$ one common mass $\mu_n>0$.  The two numbers $\mu_m$ and $\mu_n$ are
arbitrary and are not required to be equal.  Then, for every relative angle
$\phi$, every collision-free union
\begin{equation}\label{eq:mixed-polygon-relative-rotation-family}
        P_m\cup e^{i\phi}P_n
\end{equation}
is a central configuration for $\sigma=2k$.  The two polygons need not have
the same number of vertices.  For fixed $\mu_m$ and $\mu_n$, the generic
collision-free family, modulo similarities and the finite symmetries of the
two blocks, is locally one-dimensional.

\item[(ii)]
Let $X$ and $Y$ be vertex sets of Platonic solids, and let $t_X$ and $t_Y$
be their design strengths from Table~\ref{tab:platonic-design-strengths}.
Give all vertices of $X$ one common mass $\mu_X>0$ and all vertices of $Y$
one common mass $\mu_Y>0$, with no relation imposed between $\mu_X$ and
$\mu_Y$.  For every integer $k$ with
\[
        2\le k\le\min\{t_X,t_Y\},
\]
the configuration
\begin{equation}\label{eq:platonic-relative-rotation-family}
        X\cup QY,
        \qquad Q\in SO(3),
\end{equation}
is central for $\sigma=2k$ whenever the two vertex sets are disjoint.  The
two Platonic solids may be different.  For fixed $\mu_X$ and $\mu_Y$, the
generic collision-free family, modulo similarities and the finite symmetry
groups of the two blocks, is locally three-dimensional.
\end{enumerate}
\end{corollary}

Thus the relevant condition is a \emph{common lower bound on the design
strength}, not equality of the blocks.  For example, a square and a regular
pentagon may be combined on the same circle for $\sigma=4$ or $\sigma=6$,
since their strengths are $3$ and $4$, respectively; the common mass on the
square and the common mass on the pentagon may be chosen independently.
Likewise, an octahedron and a dodecahedron may be combined on the same
two-sphere for $\sigma=4$ or $\sigma=6$, since their strengths are $3$ and
$5$, and again the common vertex mass of the octahedron need not equal that
of the dodecahedron.  The phrase ``on the same sphere'' is essential: a
regular polygon regarded as a design on its intrinsic circle $S^1$ cannot be
combined by this result with a Platonic design on $S^2$, because the two
blocks are not designs of the same spherical space $S_R(V)$.

The construction does not stop with two blocks.

\begin{corollary}[Multiparameter families from several design blocks]
\label{cor:multiparameter-design-block-families}
Let $s\ge2$.

\begin{enumerate}
\item[(i)]
For each $a=1,\ldots,s$ let $P_{n_a}$ be a regular polygon on one common
circle and suppose that
\[
        n_a\ge k+1.
\]
Choose arbitrary numbers $\mu_a>0$, and give every vertex of the block
$P_{n_a}$ the common mass $\mu_a$.  The numbers $\mu_a$ may vary freely from
block to block.  For arbitrary angles $\phi_1,\ldots,\phi_s$, every
collision-free union
\begin{equation}\label{eq:many-regular-polygons}
        \bigcup_{a=1}^{s} e^{i\phi_a}P_{n_a}
\end{equation}
is a central configuration for $\sigma=2k$.  For fixed choices of the
$\mu_a$, the generic collision-free moduli space has local dimension
\begin{equation}\label{eq:polygon-many-block-dimension}
        s-1.
\end{equation}

\item[(ii)]
Let $X_1,\ldots,X_s$ be vertex sets of Platonic solids on one common
two-sphere and suppose that every $X_a$ has design strength at least $k$.
Choose arbitrary numbers $\mu_a>0$, and give every vertex of $X_a$ the
common mass $\mu_a$.  Again the numbers $\mu_a$ are independent.  Then for
arbitrary $Q_1,\ldots,Q_s\in SO(3)$ every collision-free union
\begin{equation}\label{eq:many-platonic-blocks}
        \bigcup_{a=1}^{s} Q_aX_a
\end{equation}
is a central configuration for $\sigma=2k$.  For fixed choices of the
$\mu_a$, the generic collision-free moduli space has local dimension
\begin{equation}\label{eq:platonic-many-block-dimension}
        3(s-1).
\end{equation}
\end{enumerate}
\end{corollary}

The dimensions in this corollary are the special cases $d=2$ and $d=3$ of
\qref{eq:general-rotated-block-moduli-dimension}.  More generally, for fixed
$d$ and fixed exponent $\sigma=2k$, the dimension of the resulting moduli is
not bounded independently of the number of particles: by increasing the
number of independently rotated design blocks it grows like
\[
        (s-1)\frac{d(d-1)}2.
\]

There is also a useful geometric consequence in a different direction.

\begin{corollary}[Arbitrarily dense finite central configurations]
\label{cor:dense-central-configurations}
Fix $d\ge2$, $k\ge2$, and a finite spherical $k$-design
$X\subset S_R(V)$.  For every $\varepsilon>0$ there exist finitely many
rotations
\[
        Q_1,\ldots,Q_s\in SO(V)
\]
such that the union
\[
        Q_1X\cup\cdots\cup Q_sX
\]
is collision-free, is an $\varepsilon$-net of $S_R(V)$, and is a central
configuration for the exponent $\sigma=2k$.
\end{corollary}

Thus the finite central configurations produced by the design mechanism may
become arbitrarily dense on the underlying sphere without changing the
homogeneity exponent.

For the basic planar counterexample to finiteness it is enough to return to
two equal blocks.  To obtain a uniform series indexed directly by $k$, and at
the same time to keep the number of vertices minimal within this
regular-polygon construction, we choose the smallest regular polygon whose
strength reaches $k$, namely $P_{k+1}$.  Let
\begin{equation}\label{eq:regular-kplus1-gon}
        P_{k+1}
        =
        \left\{
        e^{2\pi i j/(k+1)}:j=0,\ldots,k
        \right\}
        \subset S^1.
\end{equation}

\begin{theorem}[Planar polygonal families]
\label{thm:planar-polygonal-families}
For every integer $k\ge2$, give equal mass $1$ to the $2(k+1)$ points of
\begin{equation}\label{eq:general-two-polygons}
        X_\phi^{(k)}
        =
        P_{k+1}\cup e^{i\phi}P_{k+1}.
\end{equation}
If
\begin{equation}\label{eq:general-polygon-collision-condition}
        \phi\not\equiv0\pmod{2\pi/(k+1)},
\end{equation}
then $X_\phi^{(k)}$ is a strictly convex central configuration for
$\sigma=2k$, with
\begin{equation}\label{eq:general-polygon-lambda}
        \lambda
        =
        -(k+1)\binom{2k}{k}.
\end{equation}
Moreover, the family contains a continuum of pairwise non-similar
configurations.  More precisely, on the interval
\begin{equation}\label{eq:general-polygon-fundamental-interval}
        0<\phi<\frac{\pi}{k+1},
\end{equation}
the quantity
\begin{equation}\label{eq:general-polygon-shape-invariant}
        J_{k+1}(X_\phi^{(k)})
        :=
        \frac{1}{2(k+1)}
        \left|
        \sum_{z\in X_\phi^{(k)}}z^{k+1}
        \right|
\end{equation}
is equal to
\begin{equation}\label{eq:general-polygon-shape-invariant-value}
        J_{k+1}(X_\phi^{(k)})
        =
        \left|
        \cos\frac{(k+1)\phi}{2}
        \right|
\end{equation}
and is strictly decreasing.  Hence, for every positive even exponent
$\sigma=2k$, $k\ge2$, there exists a fixed positive mass vector for which
there are infinitely many similarity classes of central configurations.
\end{theorem}

The equal-mass normalization in
Theorem~\ref{thm:planar-polygonal-families} is chosen only to display the
simplest fixed-mass family and the clean multiplier
\qref{eq:general-polygon-lambda}.  Already
Lemma~\ref{lem:weighted-design-closure} allows the two polygonal blocks to
carry two different common positive vertex masses while preserving the
relative-angle modulus.  The phenomenon is considerably more robust: it
persists even when the individual vertex masses are perturbed independently.

\begin{theorem}[Persistence under perturbation of the masses]
\label{thm:open-mass-persistence}
Fix $k\ge2$ and a collision-free member $X_{\phi_0}^{(k)}$ of
\qref{eq:general-two-polygons}.  There exists a neighborhood $\mathcal U$ of
the equal-mass vector
\begin{equation}\label{eq:equal-mass-vector-general}
        (1,\ldots,1)\in\mathbb R_{>0}^{2(k+1)}
\end{equation}
such that, for every mass vector in $\mathcal U$, the weighted spherical
$k$-designs near $X_{\phi_0}^{(k)}$ form, after quotienting by similarities
and by the finite permutations which preserve that mass vector, a local
moduli space of dimension one.  After the neighborhood is chosen
sufficiently small, all representatives remain collision-free and strictly
convex.  By Theorem~\ref{thm:even-power-design}, they are central
configurations for the fixed exponent $\sigma=2k$.  In particular, for
every mass vector in $\mathcal U$ there are locally infinitely many
similarity classes of central configurations.
\end{theorem}

Theorem~\ref{thm:open-mass-persistence} is the statement relevant to the
generic-finiteness question in
Problem~\ref{prob:generic-homogeneous-finiteness}.  The continuous families
occur for every mass vector in an open subset of the positive mass space,
rather than only at the equal-mass point.  In particular, the failure of
finiteness cannot be confined to a lower-dimensional exceptional subset of
mass space.

The main conclusions may now be summarized as follows.  The positive even
exponents $\sigma=2k$, $k\ge2$, admit a general spherical-design mechanism
for producing central configurations in every configuration dimension
$d\ge2$.  Classical closure of spherical designs under independent rotations
and positive blockwise rescaling combines with
Theorem~\ref{thm:even-power-design} to produce large families of central
configurations.  Theorem~\ref{thm:generic-rotated-design-moduli} shows that
the rotational parameters survive generically in the similarity quotient,
giving
\[
        (s-1)\frac{d(d-1)}2
\]
local moduli parameters for $s$ independently rotated blocks.  In the plane,
Theorem~\ref{thm:planar-polygonal-families} gives explicit fixed-mass
one-parameter families for every positive even exponent, and
Theorem~\ref{thm:open-mass-persistence} shows that this degeneracy persists
through an open set of positive mass vectors.

\medskip

We finish the Introduction by describing where these assertions are proved.

Section~2 starts from the homogeneous Newton equations and derives, in
arbitrary configuration dimension $d$, the separation between the
time-independent central-configuration equations and the time-dependent
scale and rotation variables.

Section~3 develops the spherical-design mechanism.  The proof of
Theorem~\ref{thm:even-power-design} begins with the weighted averaging
identities, passes through the polynomial form of the even-power force and
the rotational-invariance argument, and ends with the explicit computation
of the multiplier \qref{eq:general-even-lambda}.  The classical closure
Lemma~\ref{lem:weighted-design-closure} is quoted there with reference to the
design literature. Combining it with
Theorem~\ref{thm:even-power-design} gives the short proof of
Corollary~\ref{cor:relative-rotations-designs}.  The section concludes by
discussing minimal and $k$-indecomposable design blocks and by giving the
barycentric description of the admissible positive weights on a fixed
support in Proposition~\ref{prop:fixed-support-weight-polytope}.

Section~4 is devoted to the geometry of the resulting moduli.  Its proof
programme begins by showing that, for generic relative orientations, the
individual design blocks can be recovered from the metric of their union
using the internal distances of the fixed blocks.  It then proves that the map from
\[
        SO(d)^s
\]
to similarity classes is generically finite-to-one after quotienting by the
common rotation and the finite internal symmetries.  This completes the proof
of Theorem~\ref{thm:generic-rotated-design-moduli}.  The polygonal and
Platonic statements
\ref{cor:polygonal-platonic-relative-rotation-families} and
\ref{cor:multiparameter-design-block-families}, as well as
Corollary~\ref{cor:dense-central-configurations}, are then deduced from this
general result.  The section ends with the explicit invariant
\qref{eq:general-polygon-shape-invariant}, which proves the pairwise
non-similarity asserted in
Theorem~\ref{thm:planar-polygonal-families}.

Section~5 proves Theorem~\ref{thm:open-mass-persistence}.  On the circle, the
weighted $k$-design equations are rewritten as the vanishing of the Fourier
moments
\[
        \sum_jm_je^{i\ell\theta_j}=0,
        \qquad
        \ell=1,\ldots,k.
\]
At two relatively rotated regular $(k+1)$-gons the angular linearization is
diagonalized by the discrete Fourier transform.  Collision-freeness implies
that its only kernel directions are the two independent rigid rotations of
the two blocks.  After quotienting by one common rotation, the angular
Jacobian has full rank, and the parameterized implicit-function theorem gives
the required one-dimensional family for every nearby positive mass vector.
The openness of distinctness and strict convexity completes the proof.

\section{From the equations of motion to the central-configuration system}
\label{sec:homographic-reduction}

The notion of a central configuration introduced in the preceding section is
purely algebraic.  We now explain how it enters the dynamics of the homogeneous
$n$-body problem.  The discussion will be carried out for a configuration of
arbitrary intrinsic dimension $d\ge2$.  Our aim is not to classify all possible
homographic motions in higher dimension, but only to describe the branch of
such motions generated by central configurations.  For this branch the
separation between the fixed geometry of the configuration and its
time-dependent scale and orientation is complete.

\subsection{The equations of motion and the center of mass}
\label{subsec:center-of-mass-reduction}

Let
\[
        x_i=x_i(t)\in\mathbb R^D,
        \qquad i=1,\ldots,n,
\]
be the trajectories of $n$ bodies of positive masses $m_i$, and put
\begin{equation}\label{eq:time-dependent-distances}
        r_{ij}(t):=|x_i(t)-x_j(t)|.
\end{equation}
For $\sigma\ne0$ the homogeneous potential is
\begin{equation}\label{eq:homogeneous-potential-body}
        U_\sigma(x_1,\ldots,x_n)
        =\frac{\kappa}{\sigma}
        \sum_{1\le i<j\le n}m_i m_j r_{ij}^{\sigma}.
\end{equation}
Since
\begin{equation}\label{eq:gradient-pair-power}
        \nabla_{x_i}|x_i-x_j|^\sigma
        =\sigma |x_i-x_j|^{\sigma-2}(x_i-x_j),
\end{equation}
Newton's equations $m_i\ddot x_i=-\nabla_{x_i}U_\sigma$ take the form
\begin{equation}\label{eq:homogeneous-equations-body}
        \ddot x_i
        =\kappa\sum_{j\ne i}
        m_j(x_j-x_i)r_{ij}^{\sigma-2},
        \qquad i=1,\ldots,n.
\end{equation}
For $\sigma=0$ the logarithmic potential in
\qref{eq:log-potential} gives the same force law with
$r_{ij}^{\sigma-2}=r_{ij}^{-2}$.  Thus
\qref{eq:homogeneous-equations-body} may be used for all real $\sigma$, with
this logarithmic interpretation at $\sigma=0$.

Let
\begin{equation}\label{eq:center-of-mass-time-dependent}
        c(t):=\frac1M\sum_{i=1}^n m_i x_i(t),
        \qquad
        M:=\sum_{i=1}^n m_i.
\end{equation}
Multiplying \qref{eq:homogeneous-equations-body} by $m_i$ and summing over
$i$, every unordered pair $\{i,j\}$ contributes two opposite vectors.
Therefore
\begin{equation}\label{eq:center-of-mass-free-motion}
        M\ddot c(t)=\sum_{i=1}^n m_i\ddot x_i(t)=0,
\end{equation}
and hence
\begin{equation}\label{eq:center-of-mass-linear-motion}
        c(t)=c_0+vt
\end{equation}
for constant vectors $c_0,v\in\mathbb R^D$.  The center of mass therefore
moves freely and can be separated from the internal motion.

Passing to the relative coordinates
\begin{equation}\label{eq:relative-coordinates}
        y_i(t):=x_i(t)-c(t)
\end{equation}
does not change any pair difference.  The variables $y_i$ satisfy the same
system \qref{eq:homogeneous-equations-body}, together with
\begin{equation}\label{eq:centered-relative-coordinates}
        \sum_{i=1}^n m_i y_i(t)=0.
\end{equation}
We shall consequently perform the internal reduction in the center-of-mass
frame and add the free translation \qref{eq:center-of-mass-linear-motion}
back at the end.

\subsection{The fixed central shape}
\label{subsec:fixed-central-shape}

Let
\[
        q=(q_1,\ldots,q_n)
\]
be a fixed non-collision configuration centered at the origin,
\begin{equation}\label{eq:centered-reference-shape}
        \sum_{i=1}^n m_iq_i=0.
\end{equation}
Let
\begin{equation}\label{eq:intrinsic-space-body}
        V:=\Span\{q_1,\ldots,q_n\},
        \qquad
        d:=\dim V=\dim q.
\end{equation}
The dynamics generated by this shape may be studied intrinsically in the
Euclidean space $V$; a fixed isometric embedding of $V$ into the ambient
$\mathbb R^D$ has no effect on the equations.  We therefore regard the
$q_i$ as vectors in the $d$-dimensional Euclidean space $V$.

Define the force field of the fixed shape by
\begin{equation}\label{eq:reference-force-field}
        F_i^{(\sigma)}(q)
        :=\sum_{j\ne i}
        m_j(q_j-q_i)|q_j-q_i|^{\sigma-2}.
\end{equation}
The configuration is central precisely when there is a real number $\lambda$
such that
\begin{equation}\label{eq:central-force-equation-body}
        F_i^{(\sigma)}(q)=\lambda q_i,
        \qquad i=1,\ldots,n.
\end{equation}
Together with the centering condition
\qref{eq:centered-reference-shape}, this gives the time-independent algebraic
system
\begin{equation}\label{eq:central-configuration-system-body}
\boxed{
\begin{aligned}
        &\sum_{i=1}^n m_iq_i=0,\\[1mm]
        &\sum_{j\ne i}m_j(q_j-q_i)|q_j-q_i|^{\sigma-2}
        =\lambda q_i,
        \qquad i=1,\ldots,n.
\end{aligned}}
\end{equation}
This is the geometric part of the problem.  It contains no time-dependent
quantity.

The sign of $\lambda$ follows directly from
\qref{eq:central-force-equation-body}.  Define the moment of inertia of the
reference shape by
\begin{equation}\label{eq:reference-moment-inertia}
        I(q):=\sum_{i=1}^n m_i|q_i|^2.
\end{equation}
Taking the scalar product of
\qref{eq:central-force-equation-body} with $m_iq_i$, summing over $i$, and
pairing the terms corresponding to $i$ and $j$, one obtains
\begin{equation}\label{eq:lambda-inertia-identity}
        \lambda I(q)
        =-\sum_{1\le i<j\le n}m_i m_j|q_i-q_j|^\sigma.
\end{equation}
For positive masses and a non-collision configuration the right-hand side is
strictly negative.  Thus
\begin{equation}\label{eq:lambda-negative}
        \lambda<0.
\end{equation}

\subsection{Scale and rotation in arbitrary dimension}
\label{subsec:scale-rotation-arbitrary-dimension}

We now keep the central shape $q$ fixed and let only its size and orientation
vary with time.  Let
\[
        \rho(t)>0,
        \qquad
        R(t)\in SO(V),
\]
where $SO(V)$ denotes the orientation-preserving orthogonal group of the
$d$-dimensional Euclidean space $V$.  We consider the motion
\begin{equation}\label{eq:homographic-ansatz-general-d}
        x_i(t)
        =c_0+vt+\rho(t)R(t)q_i.
\end{equation}
The first two terms give the free motion of the center of mass.  All internal
time dependence is contained in the scalar factor $\rho(t)$ and the rotation
$R(t)$.

For every pair $i,j$,
\begin{equation}\label{eq:homographic-pair-difference-general-d}
        x_j-x_i
        =\rho R(q_j-q_i),
\end{equation}
and hence, since $R$ is orthogonal,
\begin{equation}\label{eq:homographic-distance-scaling-general-d}
        |x_j-x_i|=\rho|q_j-q_i|.
\end{equation}
It follows from the homogeneity of the force law that
\begin{equation}\label{eq:force-homogeneity-under-homography-general-d}
        F_i^{(\sigma)}(\rho Rq)
        =\rho^{\sigma-1}R F_i^{(\sigma)}(q).
\end{equation}
For a central shape, \qref{eq:central-force-equation-body} therefore gives
\begin{equation}\label{eq:central-force-under-homography}
        F_i^{(\sigma)}(\rho Rq)
        =\lambda\rho^{\sigma-1}Rq_i.
\end{equation}
This is the first half of the separation: all geometric information has been
reduced to the single constant $\lambda$.

To differentiate the rotation, introduce its body angular velocity
\begin{equation}\label{eq:body-angular-velocity}
        \Omega(t):=R(t)^{-1}\dot R(t).
\end{equation}
Differentiating $R^TR=I$ shows that
\begin{equation}\label{eq:omega-skew}
        \Omega^T=-\Omega,
        \qquad
        \Omega(t)\in\mathfrak{so}(V).
\end{equation}
Moreover,
\begin{equation}\label{eq:Rdot-Rddot}
        \dot R=R\Omega,
        \qquad
        \ddot R=R(\dot\Omega+\Omega^2).
\end{equation}
Differentiating the internal part $\rho Rq_i$ of
\qref{eq:homographic-ansatz-general-d} twice gives
\begin{equation}\label{eq:homographic-acceleration-general-d}
        \ddot x_i
        =R\Bigl[
        \ddot\rho\,I
        +2\dot\rho\,\Omega
        +\rho\dot\Omega
        +\rho\Omega^2
        \Bigr]q_i.
\end{equation}
Combining \qref{eq:homogeneous-equations-body},
\qref{eq:central-force-under-homography}, and
\qref{eq:homographic-acceleration-general-d}, and multiplying by $R^{-1}$,
we obtain
\begin{equation}\label{eq:operator-equation-general-d-on-points}
        \Bigl[
        \ddot\rho\,I
        +2\dot\rho\,\Omega
        +\rho\dot\Omega
        +\rho\Omega^2
        \Bigr]q_i
        =\kappa\lambda\rho^{\sigma-1}q_i
\end{equation}
for every $i$.  Since the vectors $q_i$ span $V$, this is equivalent to the
operator equation
\begin{equation}\label{eq:operator-equation-general-d}
\boxed{
        \ddot\rho\,I
        +2\dot\rho\,\Omega
        +\rho\dot\Omega
        +\rho\Omega^2
        =\kappa\lambda\rho^{\sigma-1}I_V.}
\end{equation}
Equation \qref{eq:operator-equation-general-d} contains only time-dependent
quantities.  The geometric system
\qref{eq:central-configuration-system-body} and the dynamical equation
\qref{eq:operator-equation-general-d} are therefore completely separated.

\subsection{The time-dependent system}
\label{subsec:time-system-arbitrary-dimension}

The operator equation \qref{eq:operator-equation-general-d} splits naturally
into its skew-symmetric and symmetric parts.  Since $\Omega$ and
$\dot\Omega$ are skew-symmetric while $\Omega^2$ is symmetric, the
skew-symmetric part is
\begin{equation}\label{eq:omega-equation}
        2\dot\rho\,\Omega+\rho\dot\Omega=0.
\end{equation}
Multiplying by $\rho$ gives
\begin{equation}\label{eq:angular-momentum-matrix-first-integral}
        \frac{\dd}{\dd t}(\rho^2\Omega)=0.
\end{equation}
Thus there is a constant skew-symmetric endomorphism
\begin{equation}\label{eq:L-constant-skew}
        L\in\mathfrak{so}(V)
\end{equation}
such that
\begin{equation}\label{eq:rho2-omega-L}
        \rho^2\Omega=L,
        \qquad
        \Omega=\frac{L}{\rho^2}.
\end{equation}
This is the higher-dimensional analogue of conservation of angular momentum
in the planar reduction.

The symmetric part of \qref{eq:operator-equation-general-d} is
\begin{equation}\label{eq:symmetric-time-equation}
        \ddot\rho\,I+\rho\Omega^2
        =\kappa\lambda\rho^{\sigma-1}I_V.
\end{equation}
Using \qref{eq:rho2-omega-L}, this becomes
\begin{equation}\label{eq:rho-L-operator-equation}
        \ddot\rho\,I+\frac{L^2}{\rho^3}
        =\kappa\lambda\rho^{\sigma-1}I_V.
\end{equation}
All terms in \qref{eq:rho-L-operator-equation}, except possibly $L^2$, are
scalar multiples of the identity.  Consequently $L^2$ must itself be a
scalar multiple of the identity.  Because $L$ is skew-symmetric,
\[
        \langle L^2u,u\rangle=-|Lu|^2\le0,
\]
so there is a constant $h\ge0$ such that
\begin{equation}\label{eq:L-square-scalar}
\boxed{
        L^2=-h^2I_V.}
\end{equation}
Equation \qref{eq:rho-L-operator-equation} therefore reduces to the single
scalar equation
\begin{equation}\label{eq:radial-reduced-equation}
\boxed{
        \ddot\rho-\frac{h^2}{\rho^3}
        =\kappa\lambda\rho^{\sigma-1}.}
\end{equation}
Together with
\begin{equation}\label{eq:rotation-reduced-equation}
\boxed{
        R^{-1}\dot R=\frac{L}{\rho^2},
        \qquad
        L^T=-L,
        \qquad
        L^2=-h^2I_V,}
\end{equation}
this is the complete time-dependent system associated with the fixed central
shape $q$.

If $h=0$, then \qref{eq:L-square-scalar} implies $L=0$.  The rotation is
constant and the motion is purely homothetic:
\begin{equation}\label{eq:homothetic-rho-equation}
        \ddot\rho=\kappa\lambda\rho^{\sigma-1}.
\end{equation}

Suppose now that $h>0$.  Define
\begin{equation}\label{eq:J-from-L}
        J:=\frac{L}{h}.
\end{equation}
Then
\begin{equation}\label{eq:J-orthogonal-complex-structure}
        J^T=-J,
        \qquad
        J^2=-I_V.
\end{equation}
Thus $J$ is an orthogonal complex structure on $V$.  In particular, a
non-trivial rotational motion of this central type is possible only when
$d$ is even.  Introducing a scalar angle $\theta(t)$ by
\begin{equation}\label{eq:theta-dot-general-d}
        \dot\theta=\frac{h}{\rho^2},
\end{equation}
we may write
\begin{equation}\label{eq:R-isoclinic-form}
        R(t)=R_0\exp\bigl(\theta(t)J\bigr),
\end{equation}
where a constant shift of $\theta$ has been absorbed into $R_0$.
Equations \qref{eq:radial-reduced-equation} and
\qref{eq:theta-dot-general-d} are equivalently
\begin{equation}\label{eq:rho-theta-system-general-d}
\boxed{
\begin{aligned}
        \ddot\rho-\rho\dot\theta^2
        &=\kappa\lambda\rho^{\sigma-1},\\
        2\dot\rho\dot\theta+\rho\ddot\theta
        &=0.
\end{aligned}}
\end{equation}
Thus the familiar two-scalar reduction survives in every even intrinsic
dimension, but the corresponding rotation is the one-parameter subgroup
generated by a fixed orthogonal complex structure $J$.  For $d=2$ this is
the ordinary planar rotation.  For odd $d$, a full-dimensional central
configuration admits only the homothetic branch within the present
scale-and-rotation ansatz.

For $\sigma\ne0$, the radial equation
\qref{eq:radial-reduced-equation} has the first integral
\begin{equation}\label{eq:reduced-energy-first-integral}
        \frac12\dot\rho^2
        +\frac{h^2}{2\rho^2}
        -\frac{\kappa\lambda}{\sigma}\rho^\sigma
        =E_{\mathrm{red}},
\end{equation}
while for $\sigma=0$ the last term is replaced by
$-\kappa\lambda\log\rho$.

A rigidly rotating relative equilibrium occurs when $\rho(t)=\rho_0$ is
constant and $h>0$.  Then
\begin{equation}\label{eq:relative-equilibrium-frequency}
        \omega^2
        =-\kappa\lambda\rho_0^{\sigma-2},
\end{equation}
where $\omega=\dot\theta$ is constant.  Since $\lambda<0$ by
\qref{eq:lambda-negative}, such relative equilibria exist for every central
configuration in even intrinsic dimension after an orthogonal complex
structure $J$ has been chosen.  In odd intrinsic dimension, the rotational
branch is absent for a full-dimensional central configuration.

Conversely, the reduced system obtained above is sufficient for the original
equations of motion.  Indeed, suppose that $\rho(t)$ and $R(t)$ satisfy
\qref{eq:rotation-reduced-equation} and
\qref{eq:radial-reduced-equation} for constants $L$ and $h$ with
$L^2=-h^2I_V$.  With
\[
        \Omega:=R^{-1}\dot R=\frac{L}{\rho^2},
\]
we have
\[
        \dot\Omega
        =-2\frac{\dot\rho}{\rho}\Omega,
\]
and therefore
\[
        2\dot\rho\,\Omega+\rho\dot\Omega=0.
\]
Thus the skew-symmetric part of
\qref{eq:operator-equation-general-d} is satisfied.  Moreover,
\[
        \Omega^2=\frac{L^2}{\rho^4}
        =-\frac{h^2}{\rho^4}I_V,
\]
so the radial equation
\qref{eq:radial-reduced-equation} is exactly the symmetric part
\qref{eq:symmetric-time-equation}.  Adding the symmetric and
skew-symmetric parts recovers the full operator equation
\qref{eq:operator-equation-general-d}.  Since
\qref{eq:operator-equation-general-d} was obtained equivalently from the
$n$-body equations for the homographic ansatz, this proves the reverse
implication as well.

The reduction may be summarized as follows.

\begin{proposition}[Central homographic separation in arbitrary dimension]
\label{prop:central-homographic-reduction-general-d}
Let $q=(q_1,\ldots,q_n)$ be a non-collision central configuration of positive
masses, centered at the origin and spanning a Euclidean space $V$ of
dimension $d\ge2$, with central multiplier $\lambda$.  A motion of the form
\begin{equation}\label{eq:central-homographic-ansatz-proposition}
        x_i(t)=c_0+vt+\rho(t)R(t)q_i,
        \qquad
        R(t)\in SO(V),
\end{equation}
satisfies the homogeneous $n$-body equations if and only if there is a
constant $L\in\mathfrak{so}(V)$ and a constant $h\ge0$ such that
\begin{equation}\label{eq:central-homographic-time-system-proposition}
\boxed{
\begin{aligned}
        \rho^2R^{-1}\dot R&=L,\\
        L^2&=-h^2I_V,\\
        \ddot\rho-\frac{h^2}{\rho^3}
        &=\kappa\lambda\rho^{\sigma-1}.
\end{aligned}}
\end{equation}
If $h=0$, the motion is homothetic.  If $h>0$, then $d$ is even,
$J=L/h$ is an orthogonal complex structure, and the rotation has the form
\qref{eq:R-isoclinic-form}; equivalently, the time dependence is described
by the two scalar functions $\rho(t)$ and $\theta(t)$ satisfying
\qref{eq:rho-theta-system-general-d}.
\end{proposition}

\section{Spherical designs and the even-power force law}
\label{sec:spherical-design-mechanism}

Section~\ref{sec:homographic-reduction} has separated the dynamics from the
geometry.  Once a central shape has been found, its scale and rotation are
governed by the time-dependent system
\qref{eq:central-homographic-time-system-proposition}.  The geometric problem
that remains is therefore the time-independent system
\qref{eq:central-configuration-system-body}.  From now on we concentrate on
that system.

Our principal interest is in the positive even exponents
\[
        \sigma=2k,
        \qquad k=2,3,4,\ldots,
\]
although the argument below also includes $k=1$, the Hooke case.  The special
feature of these exponents is extremely elementary: when two points lie on
the same sphere, the force produced by the power $2k$ becomes a polynomial of
degree at most $k$ in the coordinates of the second point.  This is exactly
the degree controlled by a spherical $k$-design.
The main purpose of the present section is to establish the analytic
connection between spherical designs and the even-power force law and then
to combine it with a standard closure property of weighted spherical
designs.  After these results have been proved, we shall also clarify two
structural questions which are important for their later use: what should be
regarded as an elementary, or indecomposable, design block, and how much
freedom remains in the positive weights when the support of a design is held
fixed.  The separate geometric problem of determining how many genuinely
different similarity classes arise from the rotation parameters will be
treated in Section~4.

More precisely, the first result proved in the present section is the
even-power design theorem, Theorem~\ref{thm:even-power-design}.  Its proof
begins in Subsection~\ref{subsec:weighted-averages-sphere}, where the design
condition is translated into identities for weighted averages.  The
even-power force is converted into a polynomial in
Subsection~\ref{subsec:force-becomes-polynomial}; its spherical average is
shown to be radial in
Subsection~\ref{subsec:averaged-force-radial}; and the proof of
Theorem~\ref{thm:even-power-design} ends in
Subsection~\ref{subsec:design-multiplier}, where the common central
multiplier is evaluated explicitly.

The second ingredient is the standard closure property of weighted
spherical designs stated in
Lemma~\ref{lem:weighted-design-closure}.  This is a classical
design-theoretic fact and is quoted from the literature rather than reproved.
In Subsection~\ref{subsec:rotations-unions-designs} we explain its mass
interpretation and combine it with
Theorem~\ref{thm:even-power-design}.  The proof of the resulting mechanical
statement, Corollary~\ref{cor:relative-rotations-designs}, begins and ends in
that subsection.

We finish Section~3 in
Subsection~\ref{subsec:design-atoms-weight-freedom} by discussing minimal
and $k$-indecomposable design blocks and the freedom of the weights on a
fixed support.  The harmonic moment equations are there recast as a
barycentric condition after a finite-dimensional harmonic lift.  This gives
a geometric interpretation of the admissible weight polytope and separates
the fixed-support mass freedom from the moving-support persistence theorem
proved later in Section~5.

It is important to distinguish these conclusions from the moduli statements
announced in the Introduction.  Corollary~\ref{cor:relative-rotations-designs}
shows that independently rotated design blocks remain central, but it does
not by itself show that different rotations determine different similarity
classes.  The latter question requires a separate quotient argument.

Accordingly, Section~4 will prove
Theorem~\ref{thm:generic-rotated-design-moduli}. The proof will show that
for generic relative orientations the individual design blocks can be
recovered from the metric of their union using the internal distances of the
fixed blocks, and that the map from relative orientations to similarity
classes is generically finite-to-one.  The same section will then deduce
Corollaries~\ref{cor:polygonal-platonic-relative-rotation-families} and
\ref{cor:multiparameter-design-block-families}, prove
Corollary~\ref{cor:dense-central-configurations}, and complete the proof of
Theorem~\ref{thm:planar-polygonal-families} by establishing the claimed
pairwise non-similarity of its one-parameter family.

Finally, the persistence theorem
Theorem~\ref{thm:open-mass-persistence} is not a consequence of the present
section.  Its proof will be given separately in Section~5, using the Fourier
form of the weighted design equations, a rank calculation at two rotated
regular polygons, and the parameterized implicit-function theorem.

\subsection{Weighted averages on a sphere}
\label{subsec:weighted-averages-sphere}

Let $V$ be a $d$-dimensional Euclidean vector space, with $d\ge2$, and let
\begin{equation}\label{eq:sphere-body}
        S_R(V):=\{x\in V:|x|=R\}
\end{equation}
be the sphere of radius $R>0$ centered at the origin.  We denote by $d\nu_R$
the rotation-invariant probability measure on this sphere, normalized by
\begin{equation}\label{eq:sphere-measure-normalization-body}
        \int_{S_R(V)}d\nu_R=1.
\end{equation}
Thus, for any function $f$ on the sphere,
\[
        \int_{S_R(V)}f(y)\,d\nu_R(y)
\]
is the average value of $f$ over the whole sphere.

Now choose finitely many points
\[
        q_1,\ldots,q_n\in S_R(V)
\]
with positive masses $m_1,\ldots,m_n$, and put
\[
        M:=\sum_{i=1}^n m_i.
\]
For a scalar function $f$ we shall compare two averages.  The first is the
weighted discrete average over the chosen points,
\begin{equation}\label{eq:discrete-weighted-average}
        \langle f\rangle_q
        :=\frac1M\sum_{i=1}^n m_i f(q_i),
\end{equation}
and the second is the spherical average
\begin{equation}\label{eq:spherical-average}
        \langle f\rangle_{S_R}
        :=\int_{S_R(V)}f(y)\,d\nu_R(y).
\end{equation}

Definition~\ref{def:weighted-spherical-k-design} says precisely that the
weighted set $(q_i,m_i)$ is a spherical $k$-design when
\begin{equation}\label{eq:k-design-average-equality-body}
        \langle p\rangle_q
        =\langle p\rangle_{S_R}
\end{equation}
for every real polynomial $p:V\to\mathbb R$ of total degree at most $k$.
There is no approximation in \qref{eq:k-design-average-equality-body}.  For
all polynomials up to the prescribed degree, the finite weighted sum gives
exactly the same answer as averaging over every point of the sphere.  In this
sense a spherical design is an exact finite quadrature rule for low-degree
polynomials on a sphere.

The same statement may be used for vector-valued polynomials.  If
\[
        P=(P^1,\ldots,P^d):V\longrightarrow V
\]
has scalar components of degree at most $k$, then we apply
\qref{eq:k-design-average-equality-body} separately to every component and
obtain
\begin{equation}\label{eq:vector-design-average}
        \frac1M\sum_{i=1}^n m_iP(q_i)
        =\int_{S_R(V)}P(y)\,d\nu_R(y).
\end{equation}
This elementary observation will be used for the force itself, which is a
vector.

It is useful first to see what the lowest-degree conditions mean.  Suppose
that $k\ge1$.  Taking the linear polynomial
\[
        p_u(y):=u\cdot y,
        \qquad u\in V,
\]
and using the fact that the average of $y$ over a centered sphere is zero,
we find
\[
        u\cdot\left(\frac1M\sum_{i=1}^n m_iq_i\right)=0
\]
for every $u\in V$.  Hence every weighted spherical $k$-design with $k\ge1$
is automatically centered:
\begin{equation}\label{eq:design-first-moment}
        \sum_{i=1}^n m_iq_i=0.
\end{equation}
Thus the center-of-mass condition required in
\qref{eq:central-configuration-system-body} is already built into the design
property.

If $k\ge2$, quadratic polynomials give one more familiar consequence.  We
first derive explicitly the second moment of the rotation-invariant measure
on the sphere.  For $u,v\in V$, define the symmetric bilinear form
\[
        \mathcal B(u,v)
        :=
        \int_{S_R(V)}
        (u\cdot y)(v\cdot y)\,d\nu_R(y).
\]
Let $Q\in O(V)$.  Since $d\nu_R$ is invariant under orthogonal
transformations, the change of variables $y=Qz$ gives
\[
\begin{aligned}
        \mathcal B(Qu,Qv)
        &=
        \int_{S_R(V)}
        (Qu\cdot y)(Qv\cdot y)\,d\nu_R(y)\\
        &=
        \int_{S_R(V)}
        (Qu\cdot Qz)(Qv\cdot Qz)\,d\nu_R(z)\\
        &=
        \int_{S_R(V)}
        (u\cdot z)(v\cdot z)\,d\nu_R(z)\\
        &=
        \mathcal B(u,v).
\end{aligned}
\]
Thus $\mathcal B$ is invariant under the full orthogonal group $O(V)$.  The
only symmetric bilinear forms on the Euclidean space $V$ with this invariance
property are scalar multiples of the Euclidean scalar product.  Hence there
is a constant $c$ such that
\[
        \mathcal B(u,v)=c\,u\cdot v.
\]

It remains to determine $c$.  Choose an orthonormal basis
$e_1,\ldots,e_d$ of $V$.  Since $\mathcal B(u,v)=c\,u\cdot v$, we have
\[
        \mathcal B(e_a,e_b)=c\,\delta_{ab},
\]
and therefore
\[
\begin{aligned}
        dc
        &=
        \sum_{a=1}^d \mathcal B(e_a,e_a)\\
        &=
        \int_{S_R(V)}
        \sum_{a=1}^d(e_a\cdot y)^2\,d\nu_R(y)\\
        &=
        \int_{S_R(V)}
        |y|^2\,d\nu_R(y).
\end{aligned}
\]
Every point $y\in S_R(V)$ satisfies $|y|^2=R^2$, and
$d\nu_R$ is normalized by
\qref{eq:sphere-measure-normalization-body}.  Therefore
\[
        dc=R^2,
        \qquad
        c=\frac{R^2}{d}.
\]
We have consequently proved that, for arbitrary $u,v\in V$,
\begin{equation}\label{eq:sphere-second-moment}
        \int_{S_R(V)}(u\cdot y)(v\cdot y)\,d\nu_R(y)
        =\frac{R^2}{d}\,u\cdot v.
\end{equation}

Now suppose that the weighted configuration is a spherical $2$-design.
The function
\[
        p_{u,v}(y):=(u\cdot y)(v\cdot y)
\]
is a real polynomial of degree $2$.  Hence the design identity
\qref{eq:k-design-average-equality-body}, applied to $p_{u,v}$ and combined
with \qref{eq:sphere-second-moment}, gives
\begin{equation}\label{eq:design-second-moment}
        \frac1M\sum_{i=1}^n
        m_i(u\cdot q_i)(v\cdot q_i)
        =\frac{R^2}{d}\,u\cdot v.
\end{equation}

In particular, the points are not merely centered: their quadratic moment is
isotropic.  These first two moment identities already explain geometrically
why regular polygons and the Platonic solids naturally enter the problem.
The full $k$-design condition continues the same idea to all polynomial
moments of degree at most $k$.

\subsection{Why the force becomes a polynomial}
\label{subsec:force-becomes-polynomial}

We now turn to the exponent
\begin{equation}\label{eq:sigma-equals-2k-body}
        \sigma=2k,
        \qquad k\ge1.
\end{equation}
Fix a point $x\in S_R(V)$.  For another point $y\in S_R(V)$ define the
vector-valued function
\begin{equation}\label{eq:Pkx-definition}
        P_{k,x}(y)
        :=(y-x)|y-x|^{2k-2},
\end{equation}
with the value at $y=x$ understood to be $0$; for $k=1$ this is simply
$P_{1,x}(y)=y-x$.  This is exactly the vector occurring in the force law at
exponent $2k$.
Because $x$ and $y$ lie on the same sphere,
\begin{equation}\label{eq:equal-radius-distance-identity}
\begin{aligned}
        |y-x|^2
        &=|y|^2+|x|^2-2x\cdot y\\
        &=2(R^2-x\cdot y).
\end{aligned}
\end{equation}
Substituting this into \qref{eq:Pkx-definition} gives
\begin{equation}\label{eq:Pkx-polynomial-form}
\boxed{
        P_{k,x}(y)
        =2^{k-1}(y-x)(R^2-x\cdot y)^{k-1}.}
\end{equation}
For fixed $x$, the factor $R^2-x\cdot y$ is linear in $y$.  Hence
$(R^2-x\cdot y)^{k-1}$ has degree $k-1$, and multiplication by $y-x$ raises
the degree by at most one.  Consequently every scalar component of
$P_{k,x}(y)$ is a polynomial in $y$ of degree at most $k$.

This is the point at which the assumption that the exponent is a positive
even integer enters the proof.  For a general exponent the corresponding
factor is no longer a polynomial of bounded degree in $y$.  Consequently,
a finite-strength spherical-design identity does not by itself reproduce
the corresponding spherical average exactly by the argument used here.

Suppose now that $(q_i,m_i)$ is a weighted spherical $k$-design.  Fix one of
its points, say $x=q_i$.  Since the term with $j=i$ is zero, the force at
$q_i$ may be written with the sum extended over all $j$:
\begin{equation}\label{eq:design-force-as-P-sum}
\begin{aligned}
        F_i^{(2k)}(q)
        &=\sum_{j\ne i}m_j(q_j-q_i)|q_j-q_i|^{2k-2}\\
        &=\sum_{j=1}^n m_jP_{k,q_i}(q_j).
\end{aligned}
\end{equation}
The vector-valued polynomial $P_{k,q_i}$ has degree at most $k$, so
\qref{eq:vector-design-average} applies.  Therefore
\begin{equation}\label{eq:design-force-equals-sphere-average}
\boxed{
        F_i^{(2k)}(q)
        =M\int_{S_R(V)}P_{k,q_i}(y)\,d\nu_R(y).}
\end{equation}
The original finite sum has disappeared.  The force on every point of the
design is obtained from one universal spherical integral.

\subsection{Rotational invariance makes the averaged force radial}
\label{subsec:averaged-force-radial}

For $x\in S_R(V)$ put
\begin{equation}\label{eq:Ikx-definition}
        \mathcal I_k(x)
        :=\int_{S_R(V)}P_{k,x}(y)\,d\nu_R(y).
\end{equation}
We claim that this vector must be parallel to $x$.

Let $A\in O(V)$ be any orthogonal transformation satisfying $Ax=x$.  Since
$A$ preserves lengths and scalar products,
\[
        P_{k,x}(Ay)=A P_{k,x}(y).
\]
The measure $d\nu_R$ is invariant under $A$, and hence
\begin{equation}\label{eq:Ikx-stabilizer-invariance}
\begin{aligned}
        A\mathcal I_k(x)
        &=\int_{S_R(V)}A P_{k,x}(y)\,d\nu_R(y)\\
        &=\int_{S_R(V)}P_{k,x}(Ay)\,d\nu_R(y)\\
        &=\mathcal I_k(x).
\end{aligned}
\end{equation}
Thus $\mathcal I_k(x)$ is fixed by every orthogonal transformation that fixes
$x$.  The only vectors with this property are the multiples of $x$.  Indeed,
if a vector had a nonzero component perpendicular to $x$, an orthogonal
reflection fixing $x$ could reverse that perpendicular component.

We therefore have
\[
        \mathcal I_k(x)=\alpha(x)\,x
\]
for some scalar $\alpha(x)$ which, at this stage, could in principle depend
on the point $x$.

We now show that it does not.  Let $Q\in O(V)$ be arbitrary.  Since $Q$
preserves distances,
\begin{equation}\label{eq:Pk-equivariance}
        P_{k,Qx}(Qy)
        =
        QP_{k,x}(y).
\end{equation}
Using the invariance of $d\nu_R$ under $Q$ and making the change of variable
$z=Qy$, we obtain
\begin{align}
        \mathcal I_k(Qx)
        &=
        \int_{S_R(V)}P_{k,Qx}(z)\,d\nu_R(z) \notag\\
        &=
        \int_{S_R(V)}P_{k,Qx}(Qy)\,d\nu_R(y) \notag\\
        &=
        Q\int_{S_R(V)}P_{k,x}(y)\,d\nu_R(y) \notag\\
        &=
        Q\mathcal I_k(x).
        \label{eq:Ik-equivariance}
\end{align}
Thus the averaged force is $O(V)$-equivariant.

Now let $x,x'\in S_R(V)$.  Since the orthogonal group acts transitively on
the sphere, there exists $Q\in O(V)$ such that $Qx=x'$.  Using
\qref{eq:Ik-equivariance}, we find
\[
        \alpha(x')x'
        =
        \mathcal I_k(x')
        =
        \mathcal I_k(Qx)
        =
        Q\mathcal I_k(x)
        =
        \alpha(x)Qx
        =
        \alpha(x)x'.
\]
Since $x'\ne0$, it follows that
\[
        \alpha(x')=\alpha(x).
\]
Hence the scalar is independent of the point of the sphere.  We denote its
common value by $\alpha_{d,k,R}$ and obtain
\begin{equation}\label{eq:Ikx-radial}
\boxed{
        \mathcal I_k(x)=\alpha_{d,k,R}\,x
        \qquad\text{for every }x\in S_R(V).}
\end{equation}

Combining \qref{eq:design-force-equals-sphere-average} and
\qref{eq:Ikx-radial} gives immediately
\begin{equation}\label{eq:design-centrality-before-coefficient}
        F_i^{(2k)}(q)
        =M\alpha_{d,k,R}q_i,
        \qquad i=1,\ldots,n.
\end{equation}
Together with the centering relation \qref{eq:design-first-moment}, this
already proves the centrality assertion of
Theorem~\ref{thm:even-power-design}.  It remains only to compute the common
multiplier.

At this point the essential mechanism has been established.  A weighted
spherical $k$-design converts the discrete even-power force into the spherical
average \qref{eq:design-force-equals-sphere-average}, and rotational invariance
forces that average to be radial.  The design is therefore a central
configuration for the exponent $\sigma=2k$.  This completes the geometric
part of the proof of Theorem~\ref{thm:even-power-design}.  It remains only to
evaluate the scalar $\alpha_{d,k,R}$ in \qref{eq:Ikx-radial}.

\subsection{The central multiplier and completion of the proof}
\label{subsec:design-multiplier}

We now compute the scalar in \qref{eq:Ikx-radial}.  Fix
$x\in S_R(V)$ and write
\[
        e:=\frac{x}{R},
        \qquad |e|=1.
\]
Since $\mathcal I_k(x)=\alpha_{d,k,R}x$, taking the scalar product with $x$
gives
\begin{equation}\label{eq:alpha-from-radial-component}
        \alpha_{d,k,R}R^2
        =x\cdot\mathcal I_k(x).
\end{equation}
Using the definition \qref{eq:Ikx-definition} and putting
\[
        t:=\frac{x\cdot y}{R^2}=e\cdot\frac{y}{R},
        \qquad -1\le t\le1,
\]
we obtain
\begin{align}
 x\cdot P_{k,x}(y)
 &= (x\cdot y-R^2)|y-x|^{2k-2} \notag\\
 &= R^2(t-1)\bigl(2R^2(1-t)\bigr)^{k-1} \notag\\
 &= -2^{k-1}R^{2k}(1-t)^k.
 \label{eq:radial-component-Pkx}
\end{align}
Therefore
\begin{equation}\label{eq:alpha-spherical-moment}
        \alpha_{d,k,R}
        =-2^{k-1}R^{2k-2}
        \int_{S_R(V)}(1-t)^k\,d\nu_R(y).
\end{equation}

It remains to evaluate this one-dimensional spherical moment.  Put
\[
        a:=\frac{d-1}{2}.
\]
We first justify the reduction from the spherical integral to an integral in
the single variable
\[
        t=e\cdot\frac{y}{R}.
\]
For fixed $t\in(-1,1)$, the points of $S_R(V)$ having this value of $t$ form
a $(d-2)$-sphere of radius
\[
        R\sqrt{1-t^2}
\]
in the affine hyperplane perpendicular to $e$.  More explicitly, such a
point may be written as
\[
        y
        =
        R\bigl(
        te+\sqrt{1-t^2}\,\xi
        \bigr),
        \qquad
        \xi\in S_1(e^\perp).
\]
The corresponding spherical surface element is
\[
        dS(y)
        =
        R^{d-1}
        (1-t^2)^{(d-3)/2}\,
        dt\,dS_{d-2}(\xi).
\]
Indeed, the $(d-2)$ directions tangent to the slice contribute the factor
\[
        \bigl(R\sqrt{1-t^2}\bigr)^{d-2},
\]
whereas differentiation in the $t$ direction has length
\[
        \frac{R}{\sqrt{1-t^2}}.
\]
Their product is
\[
        R^{d-1}(1-t^2)^{(d-3)/2}.
\]

Consequently, after integration over the transverse sphere
$S_1(e^\perp)$, every spherical integral of a function depending only on
$t$ is proportional to
\[
        \int_{-1}^{1}
        f(t)(1-t^2)^{(d-3)/2}\,dt.
\]
Since
\[
        a-1=\frac{d-3}{2},
\]
the coordinate $t$ therefore has density proportional to
$(1-t^2)^{a-1}$ on $[-1,1]$.  Because $d\nu_R$ is normalized, the unknown
normalization constant cancels when the desired moment is divided by the
integral of the constant function $1$.  Hence
\begin{equation}\label{eq:spherical-moment-ratio}
\int_{S_R(V)}(1-t)^k\,d\nu_R(y)
=
\frac{\displaystyle
\int_{-1}^{1}(1-t)^k(1-t^2)^{a-1}\,dt}
{\displaystyle
\int_{-1}^{1}(1-t^2)^{a-1}\,dt}.
\end{equation}
Now set
\[
        u:=\frac{1-t}{2}.
\]
Then $1-t=2u$, $1+t=2(1-u)$, and both integrals in
\qref{eq:spherical-moment-ratio} become beta integrals.  Their ratio is
\begin{align}
\int_{S_R(V)}(1-t)^k\,d\nu_R(y)
&=2^k\frac{B(a+k,a)}{B(a,a)} \notag\\
&=2^k\frac{(a)_k}{(2a)_k} \notag\\
&=2^k
\frac{\bigl((d-1)/2\bigr)_k}{(d-1)_k}.
\label{eq:spherical-moment-final}
\end{align}
Here $B$ is the beta function, and in the second equality we used
\[
        \frac{B(a+k,a)}{B(a,a)}
        =\frac{(a)_k}{(2a)_k},
\]
which follows immediately from $B(p,q)=\Gamma(p)\Gamma(q)/\Gamma(p+q)$.

Substituting \qref{eq:spherical-moment-final} into
\qref{eq:alpha-spherical-moment} gives
\begin{equation}\label{eq:alpha-explicit}
\boxed{
        \alpha_{d,k,R}
        =-2^{2k-1}R^{2k-2}
        \frac{\bigl((d-1)/2\bigr)_k}{(d-1)_k}.}
\end{equation}
Finally, by \qref{eq:design-centrality-before-coefficient}, the central
multiplier is
\[
        \lambda_{d,k}=M\alpha_{d,k,R},
\]
and therefore
\[
\boxed{
        \lambda_{d,k}
        =-2^{2k-1}MR^{2k-2}
        \frac{\bigl((d-1)/2\bigr)_k}{(d-1)_k},}
\]
which is exactly \qref{eq:general-even-lambda}.  This proves the stated
formula for the multiplier and completes the proof of
Theorem~\ref{thm:even-power-design}.

\subsection{Classical closure of designs and the mechanical consequence}
\label{subsec:rotations-unions-designs}

We now recall the standard closure statement used in the relative-rotation
construction.  Lemma~\ref{lem:weighted-design-closure} is the elementary
convex-union property of weighted spherical designs, viewed as exact
cubature rules on the sphere.  In particular, if the weights of each design
are normalized to have total mass $1$, then arbitrary positive coefficients
$\beta_1,\ldots,\beta_s$ with
\[
        \beta_1+\cdots+\beta_s=1
\]
may be used to combine the designs; equivalently, before the final
normalization, every block may be multiplied by an independent positive
constant.  This weighted-union principle is stated explicitly in
\cite{MohammadpourWaldron2019}.  Orthogonal images remain designs because the
sphere and its normalized measure are invariant under $O(V)$.  We do not
repeat these standard arguments here.

For later mechanical applications it is useful to spell out the mass
interpretation once.  If an equal-weight block $X_a$ has $N_a$ vertices, we
may assign to every vertex of that block any common mass $\mu_a>0$.  Its
total mass is then
\[
        M_a=N_a\mu_a.
\]
Different blocks may have different values of $\mu_a$.  Thus neither the
classical closure lemma nor the central-configuration consequence below
requires equal masses across different polygonal or Platonic blocks.

\begin{proof}[Proof of Corollary~\ref{cor:relative-rotations-designs}]
By Lemma~\ref{lem:weighted-design-closure}, for arbitrary
$Q_1,\ldots,Q_s\in O(V)$ and arbitrary independent positive rescalings of
the block weights, the union
\[
        Q_1X_1\cup\cdots\cup Q_sX_s
\]
is a weighted spherical $k$-design on the same sphere $S_R(V)$.  Applying
Theorem~\ref{thm:even-power-design}, every collision-free such union is
therefore a central configuration for $\sigma=2k$.  This proves
Corollary~\ref{cor:relative-rotations-designs}.
\end{proof}

This completes the two principal analytic arguments of the present section:
the proof of Theorem~\ref{thm:even-power-design} and the proof of
Corollary~\ref{cor:relative-rotations-designs}.  Before turning to the
moduli problem in Section~4, we use the remaining subsection to clarify the
structure of the design blocks themselves and the freedom of their weights.

\subsection{Minimal and indecomposable design blocks; freedom of the weights}
\label{subsec:design-atoms-weight-freedom}

Corollary~\ref{cor:relative-rotations-designs} shows that a central
configuration may be assembled from several spherical $k$-design blocks, with
an independent positive rescaling of the weights on each block.  Before
turning in Section~\ref{sec:rotated-design-moduli} to the moduli created by
their relative rotations, it is useful to clarify two elementary structural
questions left open by this construction.  First, which weighted designs
should be regarded as elementary blocks, in the sense that they do not split
again into smaller designs of the same strength?  Second, if the support of a
design is held fixed, how freely may its masses be changed?

The discussion below uses the standard spherical-design framework of
\cite{DelsarteGoethalsSeidel1977,BannaiBannai2009}.  We introduce one piece
of terminology adapted specifically to the union construction used in the
present paper.

\begin{definition}[$k$-indecomposable weighted design]
\label{def:k-indecomposable}
Let
\[
        X=\{(q_i,m_i):i=1,\ldots,N\}
\]
be a positive weighted spherical $k$-design on $S_R(V)$.  We call $X$
\emph{$k$-decomposable} if its support admits a partition into two nonempty
sets,
\[
        X=X_1\sqcup X_2,
\]
such that, after retaining the relative weights inherited from $X$ and
allowing one common positive rescaling of all weights in each block, both
$X_1$ and $X_2$ are themselves weighted spherical $k$-designs on the same
sphere.  If no such partition exists, we call $X$
\emph{$k$-indecomposable}.
\end{definition}

The adjective ``indecomposable'' is used here only in the sense of
Definition~\ref{def:k-indecomposable}.  It should not be confused with the
classical notion of a \emph{rigid spherical design}.  Rigidity asks whether
the points of a design can be moved slightly, modulo orthogonal
transformations, while preserving the design property; indecomposability asks
whether one fixed weighted support splits into smaller designs.  The theory
of rigid designs, including the special classification results on the circle,
is reviewed in \cite{BannaiBannai2009}, but no rigidity assertion is used
below.

We first identify the smallest positive weighted $2$-designs.  This case is
especially relevant to the quartic exponent $\sigma=4$.

\begin{theorem}[Minimal positive weighted $2$-designs]
\label{thm:minimal-positive-weighted-2-design}
Let
\[
        q_1,\ldots,q_N\in S_R(V),
        \qquad \dim V=d\ge2,
\]
carry positive normalized weights
\[
        p_i>0,
        \qquad
        \sum_{i=1}^N p_i=1,
\]
and suppose that they form a weighted spherical $2$-design.  Then
\begin{equation}\label{eq:minimal-two-design-node-bound}
        N\ge d+1.
\end{equation}
If equality holds, then
\begin{equation}\label{eq:minimal-two-design-equality}
        p_i=\frac1{d+1},
        \qquad
        q_i\cdot q_j=-\frac{R^2}{d}
        \quad (i\ne j).
\end{equation}
Thus equality holds precisely for the equal-weight vertices of a regular
$d$-simplex inscribed in $S_R(V)$, up to an orthogonal transformation.
\end{theorem}

\begin{proof}
The first- and second-moment identities
\qref{eq:design-first-moment} and
\qref{eq:design-second-moment}, after normalizing the total mass to one, are
equivalently
\begin{equation}\label{eq:normalized-two-design-moments}
        \sum_{i=1}^N p_iq_i=0,
        \qquad
        \sum_{i=1}^N p_iq_iq_i^T
        =\frac{R^2}{d}I_V.
\end{equation}
Here $q_iq_i^T$ denotes the rank-one operator
\[
        u\longmapsto (q_i\cdot u)q_i.
\]
Choose an orthonormal basis of $V$ and let $A$ be the $d\times N$ matrix
whose $i$th column is
\begin{equation}\label{eq:minimal-two-design-column}
        a_i:=\frac{\sqrt{dp_i}}{R}\,q_i.
\end{equation}
The second identity in \qref{eq:normalized-two-design-moments} gives
\begin{equation}\label{eq:minimal-two-design-AAstar}
        AA^T=I_d.
\end{equation}
Hence $\operatorname{rank}A=d$.  On the other hand, if
\[
        s:=(\sqrt{p_1},\ldots,\sqrt{p_N})^T,
\]
then the first identity in \qref{eq:normalized-two-design-moments} gives
\[
        As=0.
\]
Since $s\ne0$, the kernel of $A$ is nontrivial, and therefore
$N\ge d+1$, proving \qref{eq:minimal-two-design-node-bound}.

Suppose now that $N=d+1$.  Then $\ker A$ is one-dimensional and is spanned
by $s$.  Since $\|s\|^2=\sum_i p_i=1$ and
\qref{eq:minimal-two-design-AAstar} says that the rows of $A$ are
orthonormal, $A^TA$ is the orthogonal projection onto $s^\perp$:
\begin{equation}\label{eq:minimal-two-design-projection}
        A^TA=I_{d+1}-ss^T.
\end{equation}
Comparing the diagonal entries of
\qref{eq:minimal-two-design-projection} gives
\[
        dp_i=1-p_i,
\]
hence $p_i=1/(d+1)$.  Comparing the off-diagonal entries gives
\[
        \frac{d\sqrt{p_ip_j}}{R^2}\,q_i\cdot q_j
        =-\sqrt{p_ip_j},
\]
and therefore $q_i\cdot q_j=-R^2/d$ for $i\ne j$.  This is precisely
\qref{eq:minimal-two-design-equality}.

It remains to verify the converse implicit in the word ``precisely'' in the
statement.  Suppose, therefore, that $q_1,\ldots,q_{d+1}\in S_R(V)$ satisfy
\[
        q_i\cdot q_j=-\frac{R^2}{d}
        \qquad (i\ne j),
\]
and give them the equal weights
\[
        p_i=\frac1{d+1}.
\]
These are the vertices of a regular $d$-simplex.

First,
\[
\begin{aligned}
        \left|\sum_{i=1}^{d+1}q_i\right|^2
        &=
        \sum_{i=1}^{d+1}|q_i|^2
        +
        2\sum_{1\le i<j\le d+1}q_i\cdot q_j\\
        &=
        (d+1)R^2
        -
        2\binom{d+1}{2}\frac{R^2}{d}\\
        &=0.
\end{aligned}
\]
Hence
\begin{equation}\label{eq:regular-simplex-first-moment}
        \sum_{i=1}^{d+1}p_iq_i=0.
\end{equation}

Next put
\[
        T:=\sum_{i=1}^{d+1}q_iq_i^T.
\]
Using the preceding centering identity, for every vertex $q_j$ we obtain
\[
\begin{aligned}
        Tq_j
        &=
        \sum_{i=1}^{d+1}q_i(q_i\cdot q_j)\\
        &=
        R^2q_j
        -
        \frac{R^2}{d}
        \sum_{i\ne j}q_i\\
        &=
        R^2q_j+\frac{R^2}{d}q_j\\
        &=
        \frac{d+1}{d}R^2q_j.
\end{aligned}
\]
To see that the vertices span $V$, consider their $(d+1)\times(d+1)$ Gram
matrix
\[
        G=(q_i\cdot q_j)_{i,j=1}^{d+1}.
\]
Using
\[
        |q_i|^2=R^2,
        \qquad
        q_i\cdot q_j=-\frac{R^2}{d}
        \quad (i\ne j),
\]
we may write
\[
        G
        =
        \frac{R^2}{d}
        \left(
        (d+1)I_{d+1}
        -
        \mathbf 1\mathbf 1^T
        \right),
\]
where
\[
        \mathbf 1=(1,\ldots,1)^T.
\]
Now $\mathbf 1\mathbf 1^T$ has eigenvalue $d+1$ on the line spanned by
$\mathbf 1$ and eigenvalue $0$ on its orthogonal complement
$\mathbf 1^\perp$.  Hence $G$ has eigenvalue $0$ on
$\operatorname{span}\{\mathbf 1\}$ and eigenvalue
\[
        \frac{d+1}{d}R^2
\]
on the $d$-dimensional space $\mathbf 1^\perp$.  Therefore
\[
        \operatorname{rank}G=d.
\]
Since the rank of a Gram matrix is the dimension of the span of the vectors
from which it is formed, the vertices $q_1,\ldots,q_{d+1}$ span $V$.
Therefore
\[
        T=\frac{d+1}{d}R^2 I_V,
\]
and hence
\begin{equation}\label{eq:regular-simplex-second-moment}
        \sum_{i=1}^{d+1}p_iq_iq_i^T
        =
        \frac{R^2}{d}I_V.
\end{equation}

Equations
\qref{eq:regular-simplex-first-moment} and
\qref{eq:regular-simplex-second-moment} reproduce the spherical averages of
all linear and quadratic polynomials, and constants are reproduced
automatically because the weights sum to one.  Thus the equal-weight vertices
of a regular $d$-simplex form a weighted spherical $2$-design.
\end{proof}

In particular, a positive weighted $2$-design with three points on a circle
is necessarily an equal-weight equilateral triangle, and a positive weighted
$2$-design with four points on a two-sphere is necessarily an equal-weight
regular tetrahedron.  Thus the triangle and the tetrahedron used in the
quartic constructions are not merely convenient symmetric examples: they
are the unique minimal positive weighted $2$-design blocks in dimensions
$d=2$ and $d=3$.

On the circle the corresponding minimality statement holds for every
strength $k$.

\begin{theorem}[Minimal positive weighted designs on the circle]
\label{thm:minimal-circle-weighted-k-design}
Let
\[
        q_1,\ldots,q_N\in S_R(V),
        \qquad \dim V=2,
\]
with positive normalized weights $p_1,\ldots,p_N$, form a weighted spherical
$k$-design.  Then
\begin{equation}\label{eq:minimal-circle-node-bound}
        N\ge k+1.
\end{equation}
If $N=k+1$, then all weights are equal and, after an orthogonal
transformation of the circle, the support is the regular $(k+1)$-gon.
\end{theorem}

\begin{proof}
Identify $V$ with $\mathbb C$ and write
\[
        q_j=Rz_j,
        \qquad |z_j|=1.
\]
For every integer $\ell\ge1$, the normalized spherical average of $z^\ell$
on the unit circle is zero:
\[
        \int_{S^1} z^\ell\,d\nu
        =
        \frac1{2\pi}\int_0^{2\pi}e^{i\ell\theta}\,d\theta
        =0.
\]
Applying the design identity
\qref{eq:weighted-k-design-definition} to the real and imaginary parts of
$z^\ell$, for $\ell=1,\ldots,k$, therefore gives
\begin{equation}\label{eq:circle-moments-section3}
        \sum_{j=1}^N p_jz_j^\ell=0,
        \qquad
        \ell=1,\ldots,k.
\end{equation}
The equations for negative powers are the complex conjugates of
\qref{eq:circle-moments-section3}.

Let $W$ be the $(k+1)\times N$ matrix
\begin{equation}\label{eq:circle-vandermonde-matrix}
        W_{rj}:=\sqrt{p_j}\,z_j^r,
        \qquad
        r=0,\ldots,k.
\end{equation}
Using \qref{eq:circle-moments-section3}, its Gram matrix is
\begin{equation}\label{eq:circle-vandermonde-unitarity}
        WW^*=I_{k+1},
\end{equation}
because the $(r,s)$ entry is
\[
        \sum_{j=1}^N p_jz_j^{\,r-s}
        =\delta_{rs}.
\]
Hence $\operatorname{rank}W=k+1$, which proves
\qref{eq:minimal-circle-node-bound}.

If $N=k+1$, the matrix $W$ is square and unitary.  The norm of its $j$th
column gives
\[
        (k+1)p_j=1,
\]
so all weights are equal.  Orthogonality of two distinct columns gives
\[
        0
        =\sum_{r=0}^k(z_i\overline{z_j})^r.
\]
Therefore $z_i\overline{z_j}$ is a nontrivial $(k+1)$st root of unity.  Since
there are exactly $k+1$ distinct nodes, they form one rotated copy of the
regular $(k+1)$-gon.
\end{proof}

Theorem~\ref{thm:minimal-circle-weighted-k-design} explains why
$P_{k+1}$ is the canonical minimal block in
Theorem~\ref{thm:planar-polygonal-families}.  Minimality, however, should
not be confused with indecomposability.  Moreover, an indecomposable
weighted design need not have the vertex set of
a regular polygon as its support.  The following elementary example makes
this distinction explicit.

\begin{proposition}[An indecomposable weighted $2$-design not supported on a regular polygon]
\label{prop:nonregular-indecomposable-circle-design}
On the unit circle take four points at the angles
\begin{equation}\label{eq:nonregular-four-node-angles}
        \frac{\pi}{6},\quad -\frac{\pi}{6},
        \quad \pi+\beta,\quad \pi-\beta,
        \qquad
        \cos\beta=\frac1{\sqrt3},
\end{equation}
and give them respectively the normalized weights
\begin{equation}\label{eq:nonregular-four-node-weights}
        \frac15,\quad\frac15,\quad\frac3{10},\quad\frac3{10}.
\end{equation}
Then this weighted set is a spherical $2$-design, is not a regular
quadrilateral, and is $2$-indecomposable in the sense of
Definition~\ref{def:k-indecomposable}.
\end{proposition}

\begin{proof}
Writing the four nodes as $z_j\in S^1$, symmetry about the real axis makes
the imaginary parts of the first two moments vanish.  The real parts are
\begin{align}
        \sum_jp_jz_j
        &=
        2\frac15\cos\frac{\pi}{6}
        -2\frac3{10}\cos\beta
        =0,
        \label{eq:nonregular-four-node-first-moment}\\
        \sum_jp_jz_j^2
        &=
        2\frac15\cos\frac{\pi}{3}
        +2\frac3{10}\cos2\beta
        =0,
        \label{eq:nonregular-four-node-second-moment}
\end{align}
where $\cos2\beta=-1/3$. Hence the first and second nonconstant Fourier moments vanish.  To conclude
that the set is a weighted spherical $2$-design, note that on the unit circle
every real polynomial of degree at most $2$ restricts to a real linear
combination of
\[
        1,\qquad
        \Re z,\qquad
        \Im z,\qquad
        \Re z^2,\qquad
        \Im z^2.
\]
The constant function is reproduced correctly because the weights sum to
one.  Equations
\qref{eq:nonregular-four-node-first-moment} and
\qref{eq:nonregular-four-node-second-moment} say that the weighted complex
averages of $z$ and $z^2$ vanish; hence the weighted averages of their real
and imaginary parts vanish as well.  These are exactly the corresponding
spherical averages on $S^1$.  Therefore the weighted design identity holds
for every real polynomial of degree at most $2$, and the four-node set is a
weighted spherical $2$-design. It is
plainly not a regular quadrilateral.  Finally,
Theorem~\ref{thm:minimal-circle-weighted-k-design} with $k=2$ shows that
every nonempty positive weighted $2$-design on the circle contains at least
three nodes.  A four-node design therefore cannot split into two nonempty
$2$-designs.  It is $2$-indecomposable.
\end{proof}

Notice that indecomposability depends on the strength $k$.  The same support
may be decomposable at one strength and indecomposable at a higher one.
A particularly transparent example is the regular hexagon $P_6$ on the
circle.  With equal weights, $P_6$ is a spherical $5$-design.  On the other
hand, as a spherical $2$-design it decomposes as the disjoint union of the
two alternating equilateral triangles,
\[
        P_6=P_3\sqcup e^{i\pi/3}P_3,
\]
and each triangle is itself a spherical $2$-design.

At strength $5$ the situation is different.  By
Theorem~\ref{thm:minimal-circle-weighted-k-design}, every nonempty positive
weighted spherical $5$-design on the circle has at least six nodes.  If the
six-node design $P_6$ could be decomposed into two nonempty positive weighted
$5$-designs, each block would therefore have to contain at least six nodes,
which is impossible.  Hence $P_6$ is $5$-indecomposable.

The preceding four-node example shows, in addition, that the class of
$k$-indecomposable weighted designs relevant to
Corollary~\ref{cor:relative-rotations-designs} is not exhausted by designs
supported on regular polygons. The same warning is important on
$S^2$.  The Platonic solids give especially symmetric examples, but they do
not exhaust spherical designs.  Delsarte, Goethals and Seidel proved that an unweighted
spherical $5$-design on $S^2$ has at least $12$ nodes, with equality forcing
the regular icosahedron \cite{DelsarteGoethalsSeidel1977}.  Reznick
constructed spherical $5$-designs on $S^2$ with $16$, $18$, $20$, $22$,
$24$, and every number of nodes at least $26$ \cite{Reznick1995}.  In
particular, his examples with $16$, $18$, $20$, and $22$ nodes cannot split,
with their inherited equal weights, into two nonempty spherical
$5$-designs, since each block would require at least $12$ nodes.  They are
therefore $5$-indecomposable in the sense of
Definition~\ref{def:k-indecomposable}.  We make no claim here of a
classification of $k$-indecomposable designs on $S^2$.  The broader
minimal-node problem for positive weighted cubature on $S^2$ remains an
active subject; for recent results see \cite{ChengYu2026}.

There is one further point which is essential for the mechanical
interpretation.  Indecomposable support does not imply that the masses on
that support are arbitrary.  In fact, the design equations admit a useful
geometric interpretation in terms of weighted barycenters.  We introduce
this interpretation from first principles.

Let
\[
        a_1,\ldots,a_N\in\mathbb R^L
\]
be finitely many points.  Numbers
\[
        p_1,\ldots,p_N\in\mathbb R,
        \qquad
        p_1+\cdots+p_N=1,
\]
are called \emph{barycentric coordinates} of a point
$a\in\mathbb R^L$ with respect to $a_1,\ldots,a_N$ if
\begin{equation}\label{eq:barycentric-coordinate-definition}
        a=\sum_{i=1}^N p_i a_i.
\end{equation}
Thus $a$ is simply the weighted average of the points $a_i$.

If all $p_i\ge0$, the weighted average in
\qref{eq:barycentric-coordinate-definition} is called a
\emph{convex combination}.  The set of all convex combinations of
$a_1,\ldots,a_N$ is called their \emph{convex hull}.  If all $p_i>0$, we
shall say that
\qref{eq:barycentric-coordinate-definition} gives \emph{positive
barycentric coordinates}.  The points $a_i$ need not be independent in any
sense, and consequently the barycentric coordinates of a point need not be
unique.

We now apply this elementary idea to spherical designs.  Recall first the
harmonic functions which encode their moments.  A real polynomial
\[
        H:V\longrightarrow\mathbb R
\]
is called \emph{harmonic} if
\[
        \Delta H=0,
\]
where, in any orthonormal coordinates $(x^1,\ldots,x^d)$ on $V$,
\[
        \Delta
        :=
        \frac{\partial^2}{\partial(x^1)^2}
        +\cdots+
        \frac{\partial^2}{\partial(x^d)^2}
\]
is the Euclidean Laplacian.  A \emph{spherical harmonic of degree $\ell$}
on $S_R(V)$ is the restriction to the sphere of a homogeneous harmonic
polynomial of degree $\ell$.

The restrictions to the sphere of polynomials of degree at most $k$ can be
described more explicitly in terms of spherical harmonics.  The standard
harmonic decomposition of homogeneous polynomials
\cite{DelsarteGoethalsSeidel1977,BannaiBannai2009} gives
\[
        P_r(x)
        =
        H_r(x)
        +|x|^2H_{r-2}(x)
        +|x|^4H_{r-4}(x)
        +\cdots
\]
for every homogeneous polynomial $P_r$ of degree $r$, where each $H_j$ is a
homogeneous harmonic polynomial of degree $j$.  On the sphere $S_R(V)$ the
factor $|x|^2$ is equal to the constant $R^2$.  Consequently the restriction
to $S_R(V)$ of every polynomial of degree at most $k$ is a linear combination
of the constant function and spherical harmonics of degrees $1,\ldots,k$.

Every spherical harmonic of positive degree has spherical average zero.  To
see this, let $Y$ be a spherical harmonic of degree $\ell\ge1$.  By
definition, there is a homogeneous harmonic polynomial
\[
        H_\ell:V\longrightarrow\mathbb R
\]
of degree $\ell$ whose restriction to $S_R(V)$ is $Y$.  The mean-value
property for harmonic functions says that the value of $H_\ell$ at the
centre of the sphere is equal to its average over the sphere:
\[
        H_\ell(0)
        =
        \int_{S_R(V)}H_\ell(y)\,d\nu_R(y).
\]
Since $H_\ell$ is homogeneous of positive degree,
\[
        H_\ell(0)=0.
\]
Consequently
\begin{equation}\label{eq:positive-degree-harmonic-average-zero}
        \int_{S_R(V)}Y(y)\,d\nu_R(y)=0
        \qquad
        \text{for every spherical harmonic $Y$ of degree $\ell\ge1$.}
\end{equation}

The constant spherical harmonics, corresponding to degree $\ell=0$, behave
differently: the average of a constant is of course the same constant.
Thus, after the weights have been normalized by
\[
        p_1+\cdots+p_N=1,
\]
the constant component of every polynomial is automatically reproduced
correctly.  The remaining design conditions are precisely the vanishing of
the weighted moments of the spherical harmonics of positive degrees
$1,\ldots,k$.  Explicitly, for fixed nodes $q_1,\ldots,q_N$ and normalized
weights $p_i$, the weighted set is a spherical $k$-design if and only if
\begin{equation}\label{eq:harmonic-moment-characterization}
        \sum_{i=1}^N p_iY(q_i)=0
\end{equation}
for every spherical harmonic $Y$ of degree $1,\ldots,k$.  We now package all
these scalar equations into a single vector equation.

Choose a real basis
\[
        Y_1,\ldots,Y_L
\]
of the finite-dimensional vector space formed by all nonconstant spherical
harmonics of degrees $1,\ldots,k$, restricted to $S_R(V)$.  We package their
values at one point into a single vector by defining the
\emph{harmonic lift}
\begin{equation}\label{eq:harmonic-lift-definition}
        \Phi_k:S_R(V)\longrightarrow\mathbb R^L,
        \qquad
        \Phi_k(q)
        :=
        \bigl(
        Y_1(q),\ldots,Y_L(q)
        \bigr).
\end{equation}
Thus every original point $q$ on the sphere is replaced by a point
$\Phi_k(q)$ in a finite-dimensional Euclidean space which records the
values at $q$ of all the chosen nonconstant spherical harmonics through
degree $k$.

The particular basis $Y_1,\ldots,Y_L$ is not important.  Replacing it by
another basis merely applies an invertible linear transformation to all the
vectors $\Phi_k(q)$.  In particular, whether a weighted average of these
vectors is zero does not depend on the chosen basis.

Suppose now that positive masses $m_i$ are placed at fixed points
\[
        q_1,\ldots,q_N\in S_R(V),
\]
and put
\[
        M:=\sum_{i=1}^N m_i,
        \qquad
        p_i:=\frac{m_i}{M}.
\]
Then
\[
        p_i>0,
        \qquad
        \sum_{i=1}^N p_i=1.
\]
By \qref{eq:harmonic-moment-characterization}, the weighted support is a
spherical $k$-design if and only if
\begin{equation}\label{eq:harmonic-barycenter-design-condition}
\boxed{
        \sum_{i=1}^N p_i\,\Phi_k(q_i)=0.}
\end{equation}

Equation~\qref{eq:harmonic-barycenter-design-condition} has a simple
geometric meaning.  The normalized masses $p_i$ are precisely positive
barycentric coordinates of the origin with respect to the lifted points
\[
        \Phi_k(q_1),\ldots,\Phi_k(q_N)\in\mathbb R^L.
\]
Equivalently, the origin is the weighted barycenter of these lifted points.

For $k=1$ this statement reduces to the ordinary center-of-mass condition
\qref{eq:design-first-moment}.  Indeed, the degree-one spherical harmonics
are linear functions, so the first harmonic moments contain exactly the same
information as
\[
        \sum_i p_iq_i=0.
\]
Thus for $k=1$ the origin is the weighted barycenter of the original points
$q_i$.  For $k\ge2$ the idea is the same, except that the points are first
lifted to a larger space which also records their quadratic, cubic, and
higher harmonic moments up to degree $k$.

The barycentric formulation gives an especially transparent description of
the possible masses on a fixed support.

\begin{proposition}[Barycentric description of the fixed-support weights]
\label{prop:fixed-support-weight-polytope}
Fix distinct points
\[
        X=\{q_1,\ldots,q_N\}\subset S_R(V),
\]
and let $\Phi_k$ be the harmonic lift
\qref{eq:harmonic-lift-definition}.  Then the normalized positive weights
which make $X$ a weighted spherical $k$-design are exactly the positive
barycentric coordinates of the origin with respect to the lifted points
\[
        \Phi_k(q_1),\ldots,\Phi_k(q_N).
\]

Equivalently, define the $L\times N$ moment matrix
\begin{equation}\label{eq:fixed-support-moment-matrix}
        B
        :=
        \begin{pmatrix}
        Y_1(q_1) & \cdots & Y_1(q_N)\\
        \vdots   &        & \vdots\\
        Y_L(q_1) & \cdots & Y_L(q_N)
        \end{pmatrix}.
\end{equation}
Thus the $i$th column of $B$ is precisely the lifted point
$\Phi_k(q_i)$.  The associated closed normalized weight polytope is
\begin{equation}\label{eq:fixed-support-weight-polytope}
\boxed{
        \overline{\mathcal W}_k(X)
        :=
        \left\{
        (p_1,\ldots,p_N)\in\mathbb R_{\ge0}^N:
        \sum_{i=1}^N p_i=1,
        \quad
        B
        \begin{pmatrix}
        p_1\\
        \vdots\\
        p_N
        \end{pmatrix}
        =0
        \right\}.}
\end{equation}
The normalized positive weights relevant to the mechanical problem form its
positive part
\begin{equation}\label{eq:fixed-support-positive-weight-set}
        \mathcal W_k(X)
        :=\overline{\mathcal W}_k(X)\cap\mathbb R_{>0}^N.
\end{equation}

For later dimension counting it is convenient to put the normalization
condition and the harmonic moment conditions into one matrix.  Define the
$(L+1)\times N$ matrix
\begin{equation}\label{eq:fixed-support-augmented-matrix}
        C
        :=
        \begin{pmatrix}
        1 & \cdots & 1\\
        Y_1(q_1) & \cdots & Y_1(q_N)\\
        \vdots   &        & \vdots\\
        Y_L(q_1) & \cdots & Y_L(q_N)
        \end{pmatrix}.
\end{equation}
Then the defining equations for the normalized weights are simply
\begin{equation}\label{eq:fixed-support-linear-system}
        C
        \begin{pmatrix}
        p_1\\
        \vdots\\
        p_N
        \end{pmatrix}
        =
        \begin{pmatrix}
        1\\
        0\\
        \vdots\\
        0
        \end{pmatrix}.
\end{equation}
Whenever $\mathcal W_k(X)$ is nonempty,
\begin{equation}\label{eq:fixed-support-weight-dimension}
        \dim\mathcal W_k(X)
        =
        N-\operatorname{rank}C.
\end{equation}
For unnormalized masses there is one additional overall scaling direction.
\end{proposition}

\begin{proof}
The first statement is exactly
\qref{eq:harmonic-barycenter-design-condition} together with the definition
\qref{eq:barycentric-coordinate-definition} of barycentric coordinates.

Writing the vector equation
\[
        \sum_{i=1}^N p_i\Phi_k(q_i)=0
\]
in the chosen harmonic basis gives
\[
        B
        \begin{pmatrix}
        p_1\\
        \vdots\\
        p_N
        \end{pmatrix}
        =0,
\]
because the columns of $B$ are the vectors $\Phi_k(q_i)$.  Adding the normalization
\[
        p_1+\cdots+p_N=1
\]
gives the linear equalities appearing in
\qref{eq:fixed-support-weight-polytope}, or equivalently the single linear
system \qref{eq:fixed-support-linear-system}.  Intersecting this affine
solution set with $p_i\ge0$ gives the closed set
$\overline{\mathcal W}_k(X)$.  It is bounded because
$p_i\ge0$ and $\sum_i p_i=1$, and it is cut out by finitely many linear
inequalities and equalities.  Hence it is a convex polytope.  Its positive
part, where all $p_i>0$, is precisely $\mathcal W_k(X)$ in
\qref{eq:fixed-support-positive-weight-set}.

Finally, assume that $\mathcal W_k(X)$ is nonempty, as in the statement of
the proposition, and choose
\[
        p^{(0)}\in\mathcal W_k(X).
\]
In particular, $p^{(0)}$ is a positive solution of
\qref{eq:fixed-support-linear-system}.  Every other solution of that linear
system has the form
\[
        p=p^{(0)}+u,
        \qquad
        Cu=0.
\]
Thus the free infinitesimal variations of the weights are precisely the
vectors in the kernel of $C$.  By the rank-nullity theorem,
\[
        \dim\ker C=N-\operatorname{rank}C.
\]
The positivity conditions $p_i>0$ are strict inequalities and hence, near
any positive solution, do not impose further linear equations.  Therefore
the positive solution set has dimension
\[
        N-\operatorname{rank}C,
\]
which proves \qref{eq:fixed-support-weight-dimension}.
\end{proof}

Proposition~\ref{prop:fixed-support-weight-polytope} gives a precise
geometric answer to the mass question for a fixed support.  The normalized
masses are not arbitrary: they must be positive barycentric coordinates of
the origin relative to the lifted configuration
\[
        \Phi_k(q_1),\ldots,\Phi_k(q_N).
\]
In particular, for $k\ge1$ the individual masses cannot vary arbitrarily and
independently while the support is held fixed.  Indeed, already the
first-moment condition \qref{eq:design-first-moment} requires
\[
        \sum_i m_iq_i=0.
\]
If this identity were to hold for every independent choice of positive
$m_i$ on one fixed nonempty spherical support, varying one mass while keeping
the others fixed would force the corresponding $q_i$ to vanish, which is
impossible on $S_R(V)$.

The same barycentric picture also makes the blockwise mass freedom
transparent.  Suppose that the fixed support is partitioned into $s$
pairwise disjoint weighted $k$-design blocks,
\[
        X=X_1\sqcup\cdots\sqcup X_s.
\]
Normalize the weights inside each block so that their sum is one.  For every
block separately, the origin is then the weighted barycenter of its lifted
points.  If the weights in the $a$th block are all multiplied by a positive
number $\beta_a$, where
\[
        \beta_1+\cdots+\beta_s=1,
\]
the barycentric equations of the blocks add to give the barycentric equation
for their union.

Because the supports of the blocks are disjoint, changing one $\beta_a$
changes exactly the weights carried by the points of $X_a$ and cannot be
absorbed into a rescaling of another block.  Thus the parameters
\[
        (\beta_1,\ldots,\beta_s),
        \qquad
        \beta_a>0,
        \qquad
        \sum_{a=1}^s\beta_a=1,
\]
give an $(s-1)$-dimensional family of distinct normalized weight vectors.
Equivalently, one obtains $s-1$ independent normalized block-mass
parameters. This is precisely the blockwise rescaling freedom used in
Lemma~\ref{lem:weighted-design-closure} and
Corollary~\ref{cor:relative-rotations-designs}; see also
\cite{MohammadpourWaldron2019}.

Additional weight freedom is possible exactly when the origin admits further
positive barycentric representations with respect to the full collection of
lifted points.  In the matrix language, this means that the kernel of $C$
contains directions beyond those generated by independent rescaling of the
individual design blocks.

It is important to keep this fixed-support statement separate from
Theorem~\ref{thm:open-mass-persistence}.  In
Proposition~\ref{prop:fixed-support-weight-polytope} the points $q_i$ are
fixed and the masses are constrained to lie in
$\mathcal W_k(X)$.  In Theorem~\ref{thm:open-mass-persistence}, by contrast,
every individual mass is allowed to vary in an open neighborhood of the
equal-mass vector because the point positions are allowed to move
simultaneously.  The Fourier and implicit-function argument proving that
stronger moving-support statement will be given in Section~5.

This completes Section~\ref{sec:spherical-design-mechanism}.
Theorem~\ref{thm:even-power-design} proves that every weighted spherical
$k$-design is central for the even exponent $\sigma=2k$;
Corollary~\ref{cor:relative-rotations-designs} shows that such designs may be
assembled from independently rotated blocks; and
Proposition~\ref{prop:fixed-support-weight-polytope} gives a barycentric and
linear-algebraic description of the admissible masses on a fixed support.
In the next section the weights of the blocks are held fixed, and we turn to
the genuinely geometric question left open by
Corollary~\ref{cor:relative-rotations-designs}: how many of the independent
rotation parameters survive after quotienting by similarities?

\section{Rotated design blocks and their moduli}
\label{sec:rotated-design-moduli}

Section~\ref{sec:spherical-design-mechanism} established that independently
rotated spherical $k$-designs on the same sphere remain central for the
positive even exponent $\sigma=2k$.  That statement concerns centrality only.
It does not yet answer a different geometric question: when do different
choices of the relative rotations give genuinely different configurations
up to similarity?

This distinction is essential.  The group $SO(d)$ has dimension
$d(d-1)/2$, but the mere presence of $d(d-1)/2$ rotation parameters does not
prove that the corresponding configurations have that many moduli.  Several
values of the parameters could in principle describe the same geometric
configuration.  Moreover, after the blocks have been united, one must rule
out the possibility that the same finite point set admits a different
partition into copies of the original blocks.

The purpose of the present section is to settle this quotient problem.  The
proof of Theorem~\ref{thm:generic-rotated-design-moduli} begins in
Subsection~\ref{subsec:generic-block-recovery}.  There we prove that, outside
a finite union of proper exceptional sets in the rotation parameters, the
individual blocks can be recovered from the distances in their union.  The
proof of Theorem~\ref{thm:generic-rotated-design-moduli} is then completed in
Subsection~\ref{subsec:generic-moduli-proof}, where we show that two generic
unions can be similar only through a common orthogonal transformation,
finitely many internal symmetries of the blocks, and finitely many
permutations of mutually congruent blocks.  This leaves
\[
        (s-1)\frac{d(d-1)}2
\]
genuine local moduli parameters.  Later in the section we apply this theorem
to regular polygons and Platonic solids, prove the density statement of
Corollary~\ref{cor:dense-central-configurations}, and complete the proof of
Theorem~\ref{thm:planar-polygonal-families} by exhibiting an explicit shape
invariant.

\subsection{Generic separation and recovery of the blocks}
\label{subsec:generic-block-recovery}

We begin with a purely elementary observation about finite point sets on a
sphere.  Let
\[
        X_1,\ldots,X_s\subset S_R(V),
        \qquad s\ge2,
\]
be the fixed weighted spherical $k$-designs appearing in
Theorem~\ref{thm:generic-rotated-design-moduli}, where $k\ge2$ and
$\dim V=d$.  We keep all point weights fixed throughout the present section.
For a rotation tuple
\[
        \mathbf Q=(Q_1,\ldots,Q_s)\in SO(V)^s
\]
we write
\begin{equation}\label{eq:rotated-union-section4}
        \mathcal X_{\mathbf Q}
        :=Q_1X_1\cup\cdots\cup Q_sX_s.
\end{equation}

Before studying the relative rotations, we record two simple consequences of
the design property.  First, each $X_a$ is centered.  Second, each $X_a$
spans the whole space $V$.  To see the latter statement, suppose that
$u\in V$ were orthogonal to every point of $X_a$.  Since $k\ge2$, the
second-moment identity \qref{eq:design-second-moment} gives
\[
        0
        =\frac1{M_a}\sum_{x\in X_a}m_x(u\cdot x)^2
        =\frac{R^2}{d}|u|^2,
\]
where $M_a$ is the total mass of the block.  Hence $u=0$.  Thus
\begin{equation}\label{eq:each-block-spans-V}
        \Span X_a=V,
        \qquad a=1,\ldots,s.
\end{equation}
This elementary fact will later imply that every orthogonal symmetry group
of a block is finite.

We now collect all distances that occur \emph{inside} the fixed blocks.  Put
\begin{equation}\label{eq:internal-distance-set}
        \mathcal D
        :=\bigcup_{a=1}^s
        \bigl\{|x-x'|:
        x,x'\in X_a,\ x\ne x'\bigr\}.
\end{equation}
Because the blocks are finite, $\mathcal D$ is a finite subset of
$(0,2R]$.

For the rotated union \qref{eq:rotated-union-section4}, a distance between
points belonging to two different blocks has the form
\[
        |Q_ax-Q_by|,
        \qquad
        a\ne b,
        \quad x\in X_a,
        \quad y\in X_b.
\]
There are two kinds of relative orientations that we want to exclude.  The
first produces a collision,
\[
        |Q_ax-Q_by|=0.
\]
The second makes a cross-block distance accidentally equal to one of the
internal distances in $\mathcal D$,
\[
        |Q_ax-Q_by|=\delta,
        \qquad \delta\in\mathcal D.
\]
Both are nongeneric conditions.

To see this explicitly, fix $a\ne b$, $x\in X_a$, $y\in X_b$, and
$\delta\in\mathcal D\cup\{0\}$.  Consider the subset
\begin{equation}\label{eq:bad-cross-distance-set}
\begin{split}
        \mathcal B(a,b;x,y;\delta)
        :=\bigl\{(Q_1,\ldots,Q_s)\in SO(V)^s:\qquad\\
        |Q_ax-Q_by|=\delta\bigr\}.
\end{split}
\end{equation}
It is closed, since the distance is a continuous function of the matrices
$Q_a$.  More precisely, it is the zero set of the real-analytic function
\begin{equation}\label{eq:bad-cross-distance-analytic-function}
        f_{a,b;x,y;\delta}(Q_1,\ldots,Q_s)
        :=
        |Q_ax-Q_by|^2-\delta^2
\end{equation}
on $SO(V)^s$.  We claim that this function is not identically zero.

Indeed, after multiplying by $Q_a^{-1}$, the relevant distance is
\[
        |x-Gy|,
        \qquad
        G:=Q_a^{-1}Q_b\in SO(V).
\]
For $d\ge2$, the group $SO(V)$ acts transitively on $S_R(V)$.  We may
therefore choose $G$ with $Gy=x$, which gives distance $0$.  If
$\delta>0$, this already shows that
$f_{a,b;x,y;\delta}$ is not identically zero.  If $\delta=0$, choose instead
$G$ with $Gy=-x$, which gives distance $2R$, and again
$f_{a,b;x,y;\delta}$ is not identically zero.

We now use the standard identity theorem for real-analytic functions: a
real-analytic function on a connected real-analytic manifold which vanishes
on a nonempty open set must vanish identically.  Since $SO(d)$ is connected
for $d\ge2$, the product $SO(V)^s$ is connected.  Hence the zero set
\[
        f_{a,b;x,y;\delta}^{-1}(0)
        =
        \mathcal B(a,b;x,y;\delta)
\]
cannot contain a nonempty open subset.  Therefore every set
\qref{eq:bad-cross-distance-set} has empty interior.

There are only finitely many choices of $a,b,x,y$ and $\delta$.  Removing all
sets \qref{eq:bad-cross-distance-set} therefore leaves an open dense subset
\begin{equation}\label{eq:generic-rotation-set-U}
        \mathcal U\subset SO(V)^s
\end{equation}
with the following property: for every $\mathbf Q\in\mathcal U$,
\begin{equation}\label{eq:generic-no-cross-internal-distance}
\begin{gathered}
        Q_aX_a\cap Q_bX_b=\varnothing,
        \qquad a\ne b,\\
        |Q_ax-Q_by|\notin\mathcal D
        \quad
        \text{whenever }a\ne b,
        \ x\in X_a,
        \ y\in X_b.
\end{gathered}
\end{equation}
Thus the union is collision-free, and no distance between two different
blocks is equal to any distance occurring inside one of the blocks.

The usefulness of \qref{eq:generic-no-cross-internal-distance} is that the
blocks can now be recognized from the metric of their union.  Given
$\mathcal X_{\mathbf Q}$, form a graph $\Gamma_{\mathbf Q}$ whose vertices
are the points of $\mathcal X_{\mathbf Q}$ and whose edges join two distinct
vertices precisely when their Euclidean distance belongs to
$\mathcal D$:
\begin{equation}\label{eq:block-recovery-graph}
        p\sim q
        \quad\Longleftrightarrow\quad
        |p-q|\in\mathcal D.
\end{equation}
If $p$ and $q$ lie in the same rotated block $Q_aX_a$, then their distance
is an internal distance of $X_a$ and therefore belongs to $\mathcal D$.
Thus every rotated block is a complete subgraph.  On the other hand,
\qref{eq:generic-no-cross-internal-distance} says that no edge joins points
from two different blocks.  Consequently the connected components of
$\Gamma_{\mathbf Q}$ are exactly
\[
        Q_1X_1,\ldots,Q_sX_s.
\]
We have proved the following lemma.

\begin{lemma}[Generic recovery of the rotated blocks]
\label{lem:generic-block-recovery}
Under the assumptions of
Theorem~\ref{thm:generic-rotated-design-moduli}, there exists an open dense
subset $\mathcal U\subset SO(V)^s$ such that for every
$\mathbf Q\in\mathcal U$ the union $\mathcal X_{\mathbf Q}$ is
collision-free and its decomposition into the rotated blocks
\[
        Q_1X_1,\ldots,Q_sX_s
\]
is recovered as the connected-component decomposition of the metric graph
\qref{eq:block-recovery-graph}.  In particular, every Euclidean isometry
between two such generic unions must permute the rotated blocks.
\end{lemma}

The proof of Lemma~\ref{lem:generic-block-recovery} begins with the internal
distance set \qref{eq:internal-distance-set} and ends here.  The lemma is the
main geometric point needed to turn rotation parameters into genuine moduli.

\subsection{The generic quotient by similarities}
\label{subsec:generic-moduli-proof}

We now complete the proof of
Theorem~\ref{thm:generic-rotated-design-moduli}.  The first step is to
understand the finite symmetries of the individual blocks.

For a weighted block $X_a$, let
\begin{equation}\label{eq:block-symmetry-group}
        G_a
        :=\{B\in O(V):BX_a=X_a
        \text{ as a weighted point set}\}.
\end{equation}
The group $G_a$ is finite.  Indeed, by
\qref{eq:each-block-spans-V} we may choose $d$ linearly independent points
of $X_a$.  An orthogonal transformation is determined by the images of
these $d$ points, and there are only finitely many possible choices of their
images in the finite set $X_a$.

More generally, for two blocks $X_a$ and $X_b$ define
\begin{equation}\label{eq:block-isometry-set}
        \operatorname{Iso}(X_a,X_b)
        :=\{B\in O(V):BX_a=X_b
        \text{ as weighted point sets}\}.
\end{equation}
This set is either empty or finite, by the same argument.  Thus there are
only finitely many orthogonal identifications between any two fixed blocks.

Let now
\[
        \mathbf Q=(Q_1,\ldots,Q_s),
        \qquad
        \mathbf Q'=(Q'_1,\ldots,Q'_s)
\]
belong to the generic set $\mathcal U$ of
Lemma~\ref{lem:generic-block-recovery}, and suppose that the two weighted
configurations $\mathcal X_{\mathbf Q}$ and
$\mathcal X_{\mathbf Q'}$ are similar.  Thus there is a similarity
\begin{equation}\label{eq:generic-similarity-map}
        T(z)=b+\rho Az,
        \qquad
        b\in V,
        \quad \rho>0,
        \quad A\in O(V),
\end{equation}
which maps one weighted point set onto the other, allowing only permutations
that preserve the attached masses.

Both unions have center of mass equal to the origin.  Indeed, every weighted
spherical $k$-design is centered by \qref{eq:design-first-moment}, and a
finite union of centered blocks is again centered.  Taking centers of mass
in \qref{eq:generic-similarity-map} therefore gives
\[
        0=b+\rho A0,
\]
so
\begin{equation}\label{eq:generic-similarity-translation-zero}
        b=0.
\end{equation}
All points of both unions lie on the same sphere $S_R(V)$.  Hence for any
point $z$ of the first union,
\[
        R=|T(z)|=\rho|Az|=\rho R,
\]
and therefore
\begin{equation}\label{eq:generic-similarity-scale-one}
        \rho=1.
\end{equation}
Thus every similarity between two generic members of our normalized family
is in fact an orthogonal transformation
\begin{equation}\label{eq:generic-similarity-is-orthogonal}
        T(z)=Az,
        \qquad A\in O(V).
\end{equation}

Because $A$ preserves distances, it preserves the graph
\qref{eq:block-recovery-graph}.  Lemma~\ref{lem:generic-block-recovery}
therefore implies that $A$ permutes the rotated blocks.  There exists a
permutation $\pi$ of $\{1,\ldots,s\}$, compatible with the weighted block
types, such that
\begin{equation}\label{eq:block-permutation-under-similarity}
        A Q_aX_a
        =Q'_{\pi(a)}X_{\pi(a)},
        \qquad a=1,\ldots,s.
\end{equation}
For every $a$ define
\begin{equation}\label{eq:Ba-definition}
        B_a
        :=(Q'_{\pi(a)})^{-1}AQ_a.
\end{equation}
Then
\[
        B_aX_a=X_{\pi(a)},
\]
so
\begin{equation}\label{eq:Ba-finite-isometry-choice}
        B_a\in\operatorname{Iso}(X_a,X_{\pi(a)}).
\end{equation}
Each set on the right is finite.  Equation \qref{eq:Ba-definition} may be
rewritten as
\begin{equation}\label{eq:Qprime-from-Q-finite-data}
        Q'_{\pi(a)}
        =AQ_aB_a^{-1}.
\end{equation}
Consequently, after the common left multiplication by $A$ is removed, the
only remaining identifications of a generic rotation tuple are generated by
finitely many permutations $\pi$ and finitely many choices of the matrices
$B_a$.  This proves the finite-to-one assertion in
Theorem~\ref{thm:generic-rotated-design-moduli}.
It remains to count the continuous parameters.  Before quotienting, the
orientation space is
\begin{equation}\label{eq:orientation-parameter-space}
        SO(V)^s.
\end{equation}
A common rotation acts by left multiplication,
\[
        A\cdot(Q_1,\ldots,Q_s)
        =(AQ_1,\ldots,AQ_s),
        \qquad A\in SO(V).
\]
This action is free.

Before taking the quotient, let us note that the generic set
$\mathcal U$ constructed in
Lemma~\ref{lem:generic-block-recovery} is invariant under this common left
action.  Indeed, its defining conditions are the absence of collisions and
the requirement that no cross-block distance belong to the fixed internal
distance set $\mathcal D$.  Both conditions depend only on Euclidean
distances, and therefore
\[
        (Q_1,\ldots,Q_s)\in\mathcal U
        \quad\Longleftrightarrow\quad
        (AQ_1,\ldots,AQ_s)\in\mathcal U
\]
for every $A\in SO(V)$.

We may therefore remove the common rotation explicitly by choosing
$A=Q_1^{-1}$.  Every orbit has a unique representative of the form
\begin{equation}\label{eq:relative-orientation-normal-form}
        (I,Q_1^{-1}Q_2,\ldots,Q_1^{-1}Q_s).
\end{equation}
Thus the continuous relative-orientation space is naturally identified with
\begin{equation}\label{eq:relative-orientation-space}
        SO(V)^{s-1}.
\end{equation}
To make the genericity statement in relative-orientation coordinates
explicit, consider the diffeomorphism
\[
\Psi:SO(V)^s
\longrightarrow
SO(V)\times SO(V)^{s-1}
\]
defined by
\[
\Psi(Q_1,\ldots,Q_s)
=
\bigl(
Q_1,
Q_1^{-1}Q_2,\ldots,Q_1^{-1}Q_s
\bigr).
\]
Its inverse is
\[
\Psi^{-1}(A,G_2,\ldots,G_s)
=
\bigl(
A,AG_2,\ldots,AG_s
\bigr).
\]

Because $\mathcal U$ is invariant under common left multiplication, membership
in $\mathcal U$ depends only on the relative rotations
\[
        (Q_1^{-1}Q_2,\ldots,Q_1^{-1}Q_s).
\]
Hence there exists a subset
\[
        \mathcal U_{\mathrm{rel}}
        \subset SO(V)^{s-1}
\]
such that
\[
        \Psi(\mathcal U)
        =
        SO(V)\times\mathcal U_{\mathrm{rel}}.
\]

Since $\Psi$ is a diffeomorphism and $\mathcal U$ is open and dense in
$SO(V)^s$, the product
\[
        SO(V)\times\mathcal U_{\mathrm{rel}}
\]
is open and dense in
\[
        SO(V)\times SO(V)^{s-1}.
\]
It follows that
\[
        \mathcal U_{\mathrm{rel}}
\]
is open and dense in $SO(V)^{s-1}$.  Thus the generic set descends exactly
to an open dense set of relative orientations.

Since
\begin{equation}\label{eq:dimension-SOd}
        \dim SO(d)=\frac{d(d-1)}2,
\end{equation}
the relative-orientation space has dimension
\begin{equation}\label{eq:moduli-dimension-section4}
        (s-1)\frac{d(d-1)}2.
\end{equation}

It remains only to check that the finite identifications found above do not
lower this dimension.  The possible use of a common orientation-reversing
element of $O(V)$ gives only a discrete identification.  The internal block
symmetries \qref{eq:block-symmetry-group}, the inter-block identifications
\qref{eq:block-isometry-set}, and the admissible permutations are all finite
as well.  Consequently every generic similarity class is represented by
only finitely many points of the relative-orientation space.

We use here the standard elementary fact that quotienting a
finite-dimensional manifold locally by finitely many identifications does
not change its local dimension.  Concretely, around a generic relative
orientation one may choose a sufficiently small coordinate neighborhood so
that the only identifications meeting that neighborhood come from the
finite stabilizer of the chosen point.  The local quotient is therefore a
finite quotient of an open subset of
$\mathbb R^{(s-1)d(d-1)/2}$, and has the same local dimension.

Hence the generic local dimension of the similarity moduli space is exactly
\qref{eq:moduli-dimension-section4}.

This proves all three parts of
Theorem~\ref{thm:generic-rotated-design-moduli} and completes its proof.

For two blocks, one may fix the first block once and for all and write the
family as
\[
        X_1\cup QX_2,
        \qquad Q\in SO(d).
\]
The theorem then says that, for generic collision-free $Q$, the passage from
$Q$ to the similarity class of the union is finite-to-one.  Hence the family
has the full generic local dimension
\begin{equation}\label{eq:two-block-dimension-section4}
        \frac{d(d-1)}2.
\end{equation}
This is the precise statement that was not supplied by the rotation parameter
alone.  In particular, two design blocks give one genuine modulus in
$d=2$, three genuine moduli in $d=3$, six in $d=4$, and so forth.

The remainder of Section~\ref{sec:rotated-design-moduli} will specialize this
general theorem to the polygonal and Platonic families announced in the
Introduction and will then prove the explicit non-similarity and density
statements.

\subsection{Polygonal, Platonic, and higher-dimensional consequences}
\label{subsec:polygonal-platonic-higher-dimensional}

We now apply Theorem~\ref{thm:generic-rotated-design-moduli} to the concrete
families announced in the Introduction.  Their centrality has already been
proved in Section~\ref{sec:spherical-design-mechanism}: by
Corollary~\ref{cor:relative-rotations-designs}, every collision-free union of
independently rotated spherical $k$-design blocks on one common sphere is a
central configuration for the exponent $\sigma=2k$.  What remained to be
proved in Corollaries~\ref{cor:polygonal-platonic-relative-rotation-families}
and~\ref{cor:multiparameter-design-block-families} was that the relative
orientation parameters survive, generically, after quotienting by
similarities.

For regular polygonal blocks the intrinsic configuration dimension is
$d=2$.  Since
\[
        \dim SO(2)=1,
\]
Theorem~\ref{thm:generic-rotated-design-moduli} gives, for $s$ independently
rotated polygonal blocks, the generic local dimension
\begin{equation}\label{eq:polygonal-moduli-dimension-section4}
        (s-1)\dim SO(2)=s-1.
\end{equation}
In particular, two polygonal blocks have one genuine local modulus.  This
proves the dimension assertion in part~(i) of
Corollary~\ref{cor:polygonal-platonic-relative-rotation-families} and the
formula \qref{eq:polygon-many-block-dimension} in part~(i) of
Corollary~\ref{cor:multiparameter-design-block-families}.

For Platonic blocks the intrinsic dimension is $d=3$.  Since
\[
        \dim SO(3)=3,
\]
the same theorem gives
\begin{equation}\label{eq:platonic-moduli-dimension-section4}
        (s-1)\dim SO(3)=3(s-1).
\end{equation}
Thus two Platonic blocks have three genuine local moduli, and $s$ blocks
have $3(s-1)$.  This proves the dimension assertion in part~(ii) of
Corollary~\ref{cor:polygonal-platonic-relative-rotation-families} and the
formula \qref{eq:platonic-many-block-dimension} in part~(ii) of
Corollary~\ref{cor:multiparameter-design-block-families}.

In these statements the block masses are held fixed while the orientations
vary.  As explained in
Subsection~\ref{subsec:design-atoms-weight-freedom}, the common vertex mass
of one block need not equal the common vertex mass of another block.  Thus
one may first choose arbitrary positive block masses $\mu_a$, fix them, and
then vary the relative orientations.  The dimensions
\qref{eq:polygonal-moduli-dimension-section4} and
\qref{eq:platonic-moduli-dimension-section4} count only the resulting
geometric orientation moduli; they do not include the independent blockwise
mass-rescaling parameters.

The same argument is not confined to dimensions two and three.  We record
two elementary higher-dimensional families and, between them, a particularly
symmetric three-dimensional example in which two different Platonic solids
produce a three-parameter family of $32$-body central configurations.

\medskip
\noindent
\textbf{Quartic families from regular simplices.}
Let $\Delta_d\subset S_R(V)$ be the vertex set of a regular $d$-simplex,
where $\dim V=d$. With equal weights its $d+1$ vertices satisfy
\[
        \sum_{x\in\Delta_d}x=0,
        \qquad
        \frac1{d+1}\sum_{x\in\Delta_d}xx^T
        =\frac{R^2}{d}I_V.
\]
By the converse established in
Theorem~\ref{thm:minimal-positive-weighted-2-design}, these identities are
exactly the equality case of the minimal positive weighted $2$-design
theorem.  Thus the equal-weight vertex set $\Delta_d$ is a spherical
$2$-design.

Therefore, for every pair $Q_1,Q_2\in SO(V)$, every collision-free union
\begin{equation}\label{eq:two-simplex-quartic-family}
        Q_1\Delta_d\cup Q_2\Delta_d
\end{equation}
is a central configuration for $\sigma=4$.  For generic relative
orientation its similarity moduli have local dimension
\begin{equation}\label{eq:two-simplex-quartic-moduli-dimension}
        \dim SO(d)=\frac{d(d-1)}2.
\end{equation}
Thus this construction gives $2(d+1)$ particles and a genuinely
higher-dimensional family for every $d\ge2$.

If the sphere has radius $R=1$ and every vertex has mass $1$, then the total
mass of \qref{eq:two-simplex-quartic-family} is $M=2(d+1)$.  Substituting
$k=2$ into \qref{eq:general-even-lambda} gives
\begin{equation}\label{eq:two-simplex-quartic-lambda}
        \lambda
        =-\frac{4(d+1)^2}{d}.
\end{equation}
For $d=3$ this is the family of two independently rotated regular
tetrahedra.  It has eight particles, three generic local moduli, and
\begin{equation}\label{eq:two-tetrahedra-quartic-lambda}
        \lambda=-\frac{64}{3}.
\end{equation}

\medskip
\noindent
\textbf{An icosahedron--dodecahedron family.}
We next give a more substantial three-dimensional example.  Let
\[
        X_{\mathrm{ico}}\subset S_R(V)
\]
be the $12$ vertices of a regular icosahedron, and let
\[
        X_{\mathrm{dod}}\subset S_R(V)
\]
be the $20$ vertices of a regular dodecahedron, both inscribed in the same
sphere in a three-dimensional Euclidean space $V$.

The equal-weight vertex sets of the icosahedron and the dodecahedron are
both spherical $5$-designs.  Consequently, for every
\begin{equation}\label{eq:icosa-dodeca-relative-rotation}
        Q\in SO(3),
\end{equation}
and for arbitrary positive numbers
\[
        \mu_{\mathrm{ico}}>0,
        \qquad
        \mu_{\mathrm{dod}}>0,
\]
we may give every vertex of $X_{\mathrm{ico}}$ the mass
$\mu_{\mathrm{ico}}$ and every vertex of $QX_{\mathrm{dod}}$ the mass
$\mu_{\mathrm{dod}}$.  By
Lemma~\ref{lem:weighted-design-closure}, every collision-free union
\begin{equation}\label{eq:icosa-dodeca-family}
        X_Q
        :=
        X_{\mathrm{ico}}\cup QX_{\mathrm{dod}}
\end{equation}
is again a weighted spherical $5$-design.

It follows from Theorem~\ref{thm:even-power-design} that the same
$32$-particle configuration \qref{eq:icosa-dodeca-family} is central
simultaneously for each of the four positive even exponents
\begin{equation}\label{eq:icosa-dodeca-four-exponents}
        \sigma=4,\quad6,\quad8,\quad10.
\end{equation}
More explicitly, for
\[
        \sigma=2k,
        \qquad
        k=2,3,4,5,
\]
its total mass is
\begin{equation}\label{eq:icosa-dodeca-total-mass}
        M
        =
        12\mu_{\mathrm{ico}}
        +20\mu_{\mathrm{dod}},
\end{equation}
and, since $d=3$, formula
\qref{eq:general-even-lambda} reduces to
\begin{equation}\label{eq:lambda-d3-simplified}
        \lambda_{3,k}
        =
        -\frac{2^{2k-1}}{k+1}\,
        M R^{2k-2}.
\end{equation}
Thus the central multiplier of \qref{eq:icosa-dodeca-family} is
\begin{equation}\label{eq:icosa-dodeca-lambda}
\boxed{
        \lambda
        =
        -\frac{2^{2k-1}}{k+1}
        \bigl(
        12\mu_{\mathrm{ico}}
        +20\mu_{\mathrm{dod}}
        \bigr)R^{2k-2},
        \qquad
        k=2,3,4,5.}
\end{equation}

For unit radius and unit masses at all $32$ vertices, one has $M=32$, and
the four multipliers are
\begin{equation}\label{eq:icosa-dodeca-unit-lambdas}
\begin{array}{c|c}
        \sigma & \lambda\\
        \hline
        4  & -\dfrac{256}{3}\\[2mm]
        6  & -256\\[2mm]
        8  & -\dfrac{4096}{5}\\[2mm]
        10 & -\dfrac{8192}{3}.
\end{array}
\end{equation}

The relative orientation in
\qref{eq:icosa-dodeca-family} is arbitrary subject only to avoiding the
exceptional subsets defined by the finitely many possible vertex-collision
conditions. Moreover, the two blocks have
different cardinalities and hence cannot be interchanged by an isometry of
the weighted union.  Theorem~\ref{thm:generic-rotated-design-moduli}
therefore shows that, for generic $Q\in SO(3)$, the map
\[
        Q
        \longmapsto
        [X_{\mathrm{ico}}\cup QX_{\mathrm{dod}}]
\]
from relative orientations to similarity classes is locally finite-to-one.
Consequently
\begin{equation}\label{eq:icosa-dodeca-moduli-dimension}
        \dim_{\mathrm{loc}}\mathcal M=3.
\end{equation}

Thus a single pair of different Platonic design blocks gives a
three-dimensional family of $32$-body central configurations, and every
member of this family is simultaneously a central configuration for four
different homogeneous potentials.

\medskip
\noindent
\textbf{Sextic families from cross-polytopes.}
Choose an orthonormal basis $e_1,\ldots,e_d$ of $V$ and put
\begin{equation}\label{eq:cross-polytope-definition-section4}
        C_d:=\{\pm Re_1,\ldots,\pm Re_d\}\subset S_R(V).
\end{equation}

The set $C_d$ is antipodal.  Hence every odd polynomial has zero average on
$C_d$, by pairwise cancellation of the contributions of $x$ and $-x$.
In particular, all spherical harmonics of degrees $1$ and $3$ have zero
average on $C_d$, exactly as they do on the sphere.

Its quadratic moment is isotropic:
\[
        \frac1{2d}\sum_{x\in C_d}xx^T
        =\frac{R^2}{d}I_V.
\]
To see what this implies for degree $2$, let
\[
        H(y)=y^TAy
\]
be a homogeneous harmonic quadratic polynomial.  Harmonicity is equivalent
to
\[
        \operatorname{tr}A=0.
\]
Therefore
\[
\begin{aligned}
        \frac1{2d}\sum_{x\in C_d}H(x)
        &=
        \operatorname{tr}\left(
        A\,\frac1{2d}\sum_{x\in C_d}xx^T
        \right)\\
        &=
        \frac{R^2}{d}\operatorname{tr}A\\
        &=0.
\end{aligned}
\]
Thus all spherical harmonic moments of positive degrees $1,2,3$ vanish.
By the harmonic-moment characterization
\qref{eq:harmonic-moment-characterization}, the equal-weight set $C_d$ is
therefore a spherical $3$-design. It follows that every
collision-free union
\begin{equation}\label{eq:two-cross-polytope-sextic-family}
        Q_1C_d\cup Q_2C_d,
        \qquad Q_1,Q_2\in SO(V),
\end{equation}
is a central configuration for $\sigma=6$, and for generic relative
orientation its similarity moduli again have local dimension
\begin{equation}\label{eq:two-cross-polytope-sextic-moduli-dimension}
        \frac{d(d-1)}2.
\end{equation}

For $R=1$ and unit vertex masses the total mass is $M=4d$.  Substituting
$k=3$ into \qref{eq:general-even-lambda} gives
\begin{equation}\label{eq:two-cross-polytope-sextic-lambda}
        \lambda=-16(d+3).
\end{equation}
When $d=3$, the cross-polytope is the regular octahedron.  Hence two
independently rotated octahedra give twelve particles, three generic local
moduli, and
\begin{equation}\label{eq:two-octahedra-sextic-lambda}
        \lambda=-96.
\end{equation}

These examples also emphasize that the two integers appearing in the theory
play independent roles.  The design strength $k$ determines the positive
even exponent $\sigma=2k$, whereas the intrinsic dimension $d$ determines
the size of the rotation group and therefore the number of relative
orientation moduli.  In particular, neither the quartic simplex family nor
the sextic cross-polytope family is intrinsically three-dimensional, while
the icosahedron--dodecahedron example illustrates how a high-strength design
may produce the same continuous family of central configurations for several
different even exponents.

This completes the proofs of
Corollaries~\ref{cor:polygonal-platonic-relative-rotation-families}
and~\ref{cor:multiparameter-design-block-families}, including their stated
generic local dimensions.  We next turn to a different consequence of the
same rotation construction: finite central configurations which are
arbitrarily dense on the underlying sphere.

\subsection{Arbitrarily dense finite central configurations}
\label{subsec:dense-central-configurations}

We now prove Corollary~\ref{cor:dense-central-configurations}.  Recall that a
finite subset $Z\subset S_R(V)$ is called an $\varepsilon$-net of the sphere
if every point of $S_R(V)$ lies at Euclidean distance less than
$\varepsilon$ from at least one point of $Z$.  Thus the assertion is that,
for every prescribed resolution $\varepsilon>0$, one can find a finite
central configuration which is distributed over the whole sphere with gaps
smaller than $\varepsilon$.

\begin{proof}[Proof of Corollary~\ref{cor:dense-central-configurations}]
Fix
\[
        d\ge2,\qquad k\ge2,
\]
and let
\[
        X\subset S_R(V)
\]
be the finite spherical $k$-design appearing in the statement of the
corollary.  Choose one point
\[
        x_0\in X.
\]

Let $\varepsilon>0$ be given.  Since the sphere $S_R(V)$ is compact, there
exists a finite $\varepsilon/3$-net
\begin{equation}\label{eq:density-epsilon-net}
        \{y_1,\ldots,y_s\}\subset S_R(V).
\end{equation}
If necessary, enlarge this finite net by adding one further point of
$S_R(V)$.  We may therefore assume
\[
        s\ge2,
\]
so that Lemma~\ref{lem:generic-block-recovery} applies to the $s$ rotated
copies used below. 
Thus, for every $z\in S_R(V)$, there is an index $a$ such that
\begin{equation}\label{eq:density-net-approximation}
        |z-y_a|<\frac{\varepsilon}{3}.
\end{equation}

The group $SO(V)$ acts transitively on the sphere because $d\ge2$.
Consequently, for each $a=1,\ldots,s$, we may choose a rotation
\[
        Q_a^0\in SO(V)
\]
such that
\begin{equation}\label{eq:density-initial-rotations}
        Q_a^0x_0=y_a.
\end{equation}
The union
\[
        Q_1^0X\cup\cdots\cup Q_s^0X
\]
therefore contains every point $y_a$ of
\qref{eq:density-epsilon-net}, and hence is itself an
$\varepsilon/3$-net of the sphere.

This initial union need not be collision-free: a point belonging to one
rotated copy of $X$ may happen to coincide with a point belonging to another
copy.  We now remove these possible coincidences by an arbitrarily small
perturbation of the rotations.

Apply Lemma~\ref{lem:generic-block-recovery} to the $s$ identical design
blocks
\[
        X_1=\cdots=X_s=X.
\]
It gives an open dense subset
\[
        \mathcal U\subset SO(V)^s
\]
such that every rotation tuple
\[
        (Q_1,\ldots,Q_s)\in\mathcal U
\]
produces a collision-free union
\[
        Q_1X\cup\cdots\cup Q_sX.
\]
In particular, because $\mathcal U$ is dense, we may choose such a tuple
arbitrarily close to
\[
        (Q_1^0,\ldots,Q_s^0).
\]
Choose it sufficiently close that
\begin{equation}\label{eq:density-small-perturbation}
        |Q_ax_0-Q_a^0x_0|
        <
        \frac{\varepsilon}{3},
        \qquad
        a=1,\ldots,s.
\end{equation}

We claim that the perturbed union is still an $\varepsilon$-net.  Indeed,
let $z\in S_R(V)$ be arbitrary.  Choose $a$ so that
\qref{eq:density-net-approximation} holds.  Using
\qref{eq:density-initial-rotations} and
\qref{eq:density-small-perturbation}, the triangle inequality gives
\[
\begin{aligned}
        |z-Q_ax_0|
        &\le
        |z-y_a|+|y_a-Q_ax_0|\\
        &=
        |z-y_a|+|Q_a^0x_0-Q_ax_0|\\
        &<
        \frac{\varepsilon}{3}
        +\frac{\varepsilon}{3}
        <\varepsilon.
\end{aligned}
\]
Since $Q_ax_0$ belongs to the union
\[
        Q_1X\cup\cdots\cup Q_sX,
\]
every point of the sphere lies within distance $\varepsilon$ of this union.
Thus the union is a collision-free $\varepsilon$-net.

It remains only to verify centrality.  Every rotated copy $Q_aX$ is again a
spherical $k$-design, and by
Lemma~\ref{lem:weighted-design-closure} their finite union, with the weights
inherited from the copies of $X$, is again a spherical $k$-design on the same
sphere.  Theorem~\ref{thm:even-power-design} therefore implies that the
collision-free union
\[
        Q_1X\cup\cdots\cup Q_sX
\]
is a central configuration for the exponent
\[
        \sigma=2k.
\]
This proves Corollary~\ref{cor:dense-central-configurations}.
\end{proof}

The result is worth contrasting with an ordinary limiting statement.  We are
not passing to an infinite distribution of matter on the sphere: for every
fixed $\varepsilon>0$ the configuration remains finite and consists of
positive point masses.  Nevertheless, by taking sufficiently many rotated
copies of one fixed spherical $k$-design, these finite central
configurations can approximate the whole sphere with arbitrarily fine
spatial resolution while the exponent $\sigma=2k$ is kept fixed.

The argument also shows that no special symmetry of the final
$\varepsilon$-dense configuration is required.  The rotations may be chosen
in the open dense generic set of
Lemma~\ref{lem:generic-block-recovery}; the density property survives because
the required perturbation can be made arbitrarily small.

We now return to the planar case.  There the construction may be made
completely explicit, and a simple shape invariant will allow us to prove
directly that the one-parameter polygonal families announced in
Theorem~\ref{thm:planar-polygonal-families} contain a continuum of pairwise
non-similar central configurations.

\subsection{Explicit planar family and pairwise non-similarity}
\label{subsec:explicit-planar-family}

We finish the present section by completing the proof of
Theorem~\ref{thm:planar-polygonal-families}.  The centrality of the family
already follows from the spherical-design construction.  What remains is to
record the multiplier and the elementary geometry of the configuration, and
then to prove explicitly that the parameter $\phi$ produces genuinely
different similarity classes.

\begin{proof}[Proof of Theorem~\ref{thm:planar-polygonal-families}]
Put
\[
        n:=k+1.
\]
Then the family \qref{eq:general-two-polygons} may be written as
\begin{equation}\label{eq:planar-family-n-notation}
        X_\phi
        =
        P_n\cup e^{i\phi}P_n,
\end{equation}
where
\[
        P_n
        =
        \left\{
        e^{2\pi i j/n}:j=0,\ldots,n-1
        \right\}.
\]
The regular $n$-gon is a spherical $(n-1)$-design, hence a spherical
$k$-design because $n-1=k$.  By
Lemma~\ref{lem:weighted-design-closure}, the union
\qref{eq:planar-family-n-notation} is therefore a spherical $k$-design for
every relative angle $\phi$, provided that the two blocks do not collide.
Theorem~\ref{thm:even-power-design} then implies that every collision-free
member of the family is a central configuration for
\[
        \sigma=2k.
\]

The two regular $n$-gons have a common vertex precisely when
\[
        e^{i\phi}
\]
is an $n$th root of unity.  Thus the union is collision-free exactly when
\begin{equation}\label{eq:planar-family-collision-n}
        e^{in\phi}\ne1,
\end{equation}
which is equivalent to
\qref{eq:general-polygon-collision-condition}.

All $2n$ masses are equal to one and the common circle has radius $R=1$.
Hence the total mass is
\[
        M=2n=2(k+1).
\]
Substituting this into the planar multiplier formula
\qref{eq:planar-general-even-lambda} gives
\[
\begin{aligned}
        \lambda
        &=
        -\frac12
        \binom{2k}{k}\,2(k+1)\\
        &=
        -(k+1)\binom{2k}{k}\\
        &=
        -n\binom{2n-2}{n-1},
\end{aligned}
\]
which proves \qref{eq:general-polygon-lambda}.

The collision-free configurations are strictly convex.  Indeed, their
$2n$ points are distinct and lie on the same circle.  The tangent line to
the circle at any one of these points is a supporting line which meets the
circle only at that point.  Hence every point is a vertex of the convex
hull.

It remains to prove that the relative angle gives a genuine modulus modulo
similarities.  For this purpose it is useful first to introduce a
similarity-invariant version of the quantity appearing in
\qref{eq:general-polygon-shape-invariant}.

Let
\[
        Z=\{z_1,\ldots,z_N\}\subset\mathbb C
\]
be a finite equal-mass configuration, and let
\[
        c(Z):=\frac1N\sum_{j=1}^N z_j
\]
be its center of mass.  For a positive integer $n$ define
\begin{equation}\label{eq:normalized-planar-shape-invariant}
        \mathcal J_n(Z)
        :=
        \frac{
        \left|
        \displaystyle\sum_{j=1}^N
        \bigl(z_j-c(Z)\bigr)^n
        \right|}
        {
        \displaystyle\sum_{j=1}^N
        |z_j-c(Z)|^n
        }.
\end{equation}
This quantity is invariant under Euclidean similarities and under
permutations of the points.

Indeed, under an orientation-preserving similarity
\[
        z\longmapsto a+s e^{i\theta}z,
        \qquad s>0,
\]
the centered coordinates transform as
\[
        z_j-c(Z)
        \longmapsto
        s e^{i\theta}\bigl(z_j-c(Z)\bigr).
\]
The numerator of
\qref{eq:normalized-planar-shape-invariant} is therefore multiplied by
$s^n$, since the additional factor $e^{in\theta}$ disappears after taking
the absolute value, while the denominator is multiplied by the same factor
$s^n$.  Under an orientation-reversing similarity the centered coordinates are
complex conjugated up to a common unit phase, which again leaves the
absolute value of the numerator unchanged.  Thus $\mathcal J_n$ depends only on the
similarity class of the configuration.

For the family $X_\phi$ the center of mass is zero, because each regular
polygon is centered separately.  Moreover every point has modulus one.
Consequently
\begin{equation}\label{eq:J-equals-normalized-J}
        \mathcal J_n(X_\phi)
        =
        \frac1{2n}
        \left|
        \sum_{z\in X_\phi}z^n
        \right|,
\end{equation}
so that $\mathcal J_n$ is precisely the quantity denoted by $J_{k+1}$ in
\qref{eq:general-polygon-shape-invariant}.

For the first polygon every vertex satisfies $z^n=1$, and hence
\begin{equation}\label{eq:first-polygon-nth-moment}
        \sum_{z\in P_n}z^n=n.
\end{equation}
For the rotated polygon,
\[
        z\in e^{i\phi}P_n
        \quad\Longrightarrow\quad
        z^n=e^{in\phi},
\]
and therefore
\begin{equation}\label{eq:second-polygon-nth-moment}
        \sum_{z\in e^{i\phi}P_n}z^n
        =
        ne^{in\phi}.
\end{equation}
Combining
\qref{eq:J-equals-normalized-J}--\qref{eq:second-polygon-nth-moment}, we
obtain
\begin{align}
        J_{k+1}(X_\phi)
        &=
        \mathcal J_n(X_\phi)\notag\\
        &=
        \frac12\left|1+e^{in\phi}\right|\notag\\
        &=
        \left|\cos\frac{n\phi}{2}\right|\notag\\
        &=
        \left|
        \cos\frac{(k+1)\phi}{2}
        \right|.
        \label{eq:planar-family-invariant-computation}
\end{align}
This is exactly
\qref{eq:general-polygon-shape-invariant-value}.

On the interval
\[
        0<\phi<\frac{\pi}{n}
        =\frac{\pi}{k+1},
\]
the argument $n\phi/2$ lies in $(0,\pi/2)$, so the absolute value in
\qref{eq:planar-family-invariant-computation} may be omitted and
\[
        J_{k+1}(X_\phi)
        =
        \cos\frac{n\phi}{2}
\]
is strictly decreasing.  Thus two different values
\[
        \phi_1,\phi_2
        \in
        \left(0,\frac{\pi}{k+1}\right)
\]
give different values of the similarity invariant $\mathcal J_n$.
The corresponding configurations cannot therefore be similar.

Hence the family contains a continuum of pairwise non-similar, strictly
convex central configurations for one fixed positive mass vector and the
fixed exponent $\sigma=2k$.  This completes the proof of
Theorem~\ref{thm:planar-polygonal-families}.
\end{proof}

We have now completed all proofs assigned to
Section~\ref{sec:rotated-design-moduli}.  Theorem~
\ref{thm:generic-rotated-design-moduli} proves that independently rotated
design blocks give, generically,
\[
        (s-1)\frac{d(d-1)}2
\]
genuine local moduli in arbitrary intrinsic dimension $d$.  The concrete
polygonal and Platonic consequences were established in
Subsection~\ref{subsec:polygonal-platonic-higher-dimensional}, including the
higher-dimensional simplex and cross-polytope families and the
three-dimensional icosahedron--dodecahedron family.
Corollary~\ref{cor:dense-central-configurations} was proved in
Subsection~\ref{subsec:dense-central-configurations}, showing that finite
central configurations may be arbitrarily dense on the underlying sphere.
Finally, the proof above gives an explicit continuum of pairwise
non-similar planar configurations for every positive even exponent.

The remaining issue is of a different nature.  Up to this point the
relative orientations have varied while the admissible weights were held
fixed, although Section~\ref{sec:spherical-design-mechanism} described the
constrained weight freedom available on a fixed support.  In the next
section we prove, in the planar case, the stronger statement that the
continuous family persists for every mass vector in an open neighborhood of
the equal-mass vector, provided the positions of the points are allowed to
move simultaneously.

\section{Persistence under perturbations of the masses}
\label{sec:open-mass-persistence}

We now prove Theorem~\ref{thm:open-mass-persistence}.  The argument is
specific to the circle and is stronger than the fixed-support weight freedom
described in Section~\ref{sec:spherical-design-mechanism}.  Here the
individual masses may vary independently in an open neighborhood of the
equal-mass vector because the angular positions are allowed to move at the
same time.

\subsection{Fourier equations and the angular rank}
\label{subsec:circle-fourier-rank}

Write the unit circle as
\[
        S^1=\{e^{i\theta}:\theta\in\mathbb R\}.
\]
Let positive masses $m_1,\ldots,m_N$ be placed at
\[
        z_j=e^{i\theta_j}.
\]
On $S^1$ the weighted spherical $k$-design equations are equivalent to the
vanishing of the first $k$ nonconstant Fourier moments:
\begin{equation}\label{eq:circle-weighted-fourier}
\boxed{
        \sum_{j=1}^{N}m_j e^{i\ell\theta_j}=0,
        \qquad
        \ell=1,\ldots,k.}
\end{equation}
Indeed, on the unit circle every polynomial of degree at most $k$ restricts
to a linear combination of the Fourier modes
\[
        1,e^{\pm i\theta},\ldots,e^{\pm ik\theta}.
\]
The normalized spherical average of every nonconstant Fourier mode is zero:
\begin{equation}\label{eq:circle-fourier-spherical-average}
        \frac1{2\pi}\int_0^{2\pi}e^{i\ell\theta}\,d\theta=0,
        \qquad
        \ell\ne0.
\end{equation}
Its weighted discrete average is
\[
        \frac1M\sum_{j=1}^N m_j e^{i\ell\theta_j},
        \qquad
        M:=\sum_{j=1}^N m_j.
\]
Therefore the design identity for the positive Fourier modes
$\ell=1,\ldots,k$ is exactly
\qref{eq:circle-weighted-fourier}.  The equations for the negative modes
follow by complex conjugation:
\[
        \sum_{j=1}^N m_j e^{-i\ell\theta_j}
        =
        \overline{
        \sum_{j=1}^N m_j e^{i\ell\theta_j}
        }.
\]
Finally, the constant mode is reproduced automatically, since both its
weighted discrete average and its normalized spherical average are equal to
the same constant.

Thus \qref{eq:circle-weighted-fourier} is not merely necessary but also
sufficient for the weighted spherical $k$-design identities on $S^1$.

Put
\[
        n:=k+1,
\]
so that our reference configuration consists of two regular $n$-gons.
From this point on, for this $2n$-particle family, we relabel the particles
by
\[
        j=0,\ldots,2n-1.
\]
Thus we write
\[
        \theta=(\theta_0,\ldots,\theta_{2n-1}),
        \qquad
        m=(m_0,\ldots,m_{2n-1}),
\]
and define
\[
        M_\ell(\theta,m)
        :=
        \sum_{j=0}^{2n-1}m_j e^{i\ell\theta_j},
        \qquad
        \ell=1,\ldots,n-1.
\]
and let
\begin{equation}\label{eq:mass-persistence-moment-map}
        \mathcal F(\theta,m)
        :=
        \bigl(
        \Re M_1,\Im M_1,\ldots,
        \Re M_{n-1},\Im M_{n-1}
        \bigr)
        \in\mathbb R^{2n-2}.
\end{equation}
Thus
\[
        \mathcal F(\theta,m)=0
\]
is exactly the weighted spherical $k$-design condition.

Fix a collision-free member
\[
        X_{\phi_0}^{(k)}
        =
        P_n\cup e^{i\phi_0}P_n.
\]
We take the reference angles to be
\begin{equation}\label{eq:base-two-ngons-angles}
\begin{aligned}
        \theta_j^0
        &=
        \frac{2\pi j}{n},
        &&j=0,\ldots,n-1,\\
        \theta_{n+j}^0
        &=
        \phi_0+\frac{2\pi j}{n},
        &&j=0,\ldots,n-1,
\end{aligned}
\end{equation}
with all masses equal to one.  Collision-freeness is precisely
\begin{equation}\label{eq:mass-persistence-collision-condition}
        e^{in\phi_0}\ne1.
\end{equation}

The only calculation needed for the implicit-function argument is the
following rank statement.

\begin{lemma}[Angular rank at two regular polygons]
\label{lem:angular-jacobian-rank}
At the equal-mass configuration
\qref{eq:base-two-ngons-angles}, satisfying
\qref{eq:mass-persistence-collision-condition}, the derivative of
$\mathcal F$ with respect to the $2n$ angular variables has rank
\begin{equation}\label{eq:mass-persistence-rank}
        \operatorname{rank}D_\theta\mathcal F
        =
        2n-2
        =
        2k.
\end{equation}
Its kernel is two-dimensional and consists exactly of independent rigid
infinitesimal rotations of the two regular $n$-gons.
\end{lemma}

\begin{proof}
Let
\[
        u_0,\ldots,u_{n-1},
        \qquad
        v_0,\ldots,v_{n-1}
\]
be infinitesimal angular variations of the first and second polygons,
respectively, and put
\[
        \omega:=e^{2\pi i/n}.
\]
Linearizing the complex equation $M_\ell=0$ gives, apart from the nonzero
factor $i\ell$,
\begin{equation}\label{eq:angular-linearization-fourier}
        \widehat u_\ell
        +
        e^{i\ell\phi_0}\widehat v_\ell
        =0,
        \qquad
        \ell=1,\ldots,n-1,
\end{equation}
where
\[
        \widehat u_\ell
        :=
        \sum_{j=0}^{n-1}u_j\omega^{\ell j},
        \qquad
        \widehat v_\ell
        :=
        \sum_{j=0}^{n-1}v_j\omega^{\ell j}.
\]

Because the variations $u_j,v_j$ are real,
\begin{equation}\label{eq:fourier-reality}
        \widehat u_{n-\ell}
        =
        \overline{\widehat u_\ell},
        \qquad
        \widehat v_{n-\ell}
        =
        \overline{\widehat v_\ell}.
\end{equation}
Taking the complex conjugate of
\qref{eq:angular-linearization-fourier} and comparing it with the equation
having index $n-\ell$ gives
\[
        e^{-i\ell\phi_0}
        \bigl(e^{in\phi_0}-1\bigr)
        \widehat v_{n-\ell}
        =0.
\]
By \qref{eq:mass-persistence-collision-condition},
\[
        \widehat v_{n-\ell}=0.
\]
Equation \qref{eq:angular-linearization-fourier} then gives
\[
        \widehat u_\ell=0
\]
for every nonzero Fourier mode $\ell=1,\ldots,n-1$.

We now use the invertibility of the discrete Fourier transform.  With our
convention,
\[
        \widehat u_\ell
        =
        \sum_{j=0}^{n-1}u_j\omega^{\ell j},
        \qquad
        \ell=0,\ldots,n-1,
\]
the inverse transform is
\begin{equation}\label{eq:inverse-discrete-fourier-transform}
        u_j
        =
        \frac1n
        \sum_{\ell=0}^{n-1}
        \widehat u_\ell\,\omega^{-\ell j},
        \qquad
        j=0,\ldots,n-1,
\end{equation}
and similarly for $v$.

Since all nonzero Fourier modes of $u$ and $v$ vanish, only the terms
$\ell=0$ remain in the inverse transform.  Therefore
\[
        u_j=\frac{\widehat u_0}{n},
        \qquad
        v_j=\frac{\widehat v_0}{n}
        \qquad
        \text{for every }j,
\]
and hence
\[
        u_0=\cdots=u_{n-1},
        \qquad
        v_0=\cdots=v_{n-1}.
\]
These are precisely the independent infinitesimal rigid rotations of the two
polygons.  Thus
\[
        \dim\ker D_\theta\mathcal F=2.
\]
Since the angular space has dimension $2n$, rank-nullity gives
\[
        \operatorname{rank}D_\theta\mathcal F
        =
        2n-2,
\]
as claimed.
\end{proof}

\subsection{Implicit-function theorem and completion of the proof}
\label{subsec:mass-ift-completion}

We now complete the proof of
Theorem~\ref{thm:open-mass-persistence}.

\begin{proof}[Proof of Theorem~\ref{thm:open-mass-persistence}]
At the reference configuration,
Lemma~\ref{lem:angular-jacobian-rank} shows that
$D_\theta\mathcal F$ is surjective onto
$\mathbb R^{2n-2}$.  Hence one may select $2n-2$ of the angular variables
for which the corresponding square Jacobian minor is invertible.

The remaining two angular variables and all $2n$ masses may therefore be
used as parameters in the parameterized implicit-function theorem.  It
follows that, for every mass vector sufficiently close to
\[
        (1,\ldots,1)\in\mathbb R_{>0}^{2n},
\]
the solutions of
\[
        \mathcal F(\theta,m)=0
\]
near the reference configuration form a smooth two-dimensional local family
in the angular variables.

One of these two dimensions is present for an elementary reason: simultaneous
rotation of all points,
\[
        \theta_j\longmapsto\theta_j+\tau,
\]
preserves every equation
\qref{eq:circle-weighted-fourier}.  We now remove this trivial direction
explicitly rather than merely counting it.

Let $\mathcal S_m$ denote the two-dimensional local solution manifold in the
angular variables for one fixed nearby mass vector $m$.  On $\mathcal S_m$
consider the function
\[
        g(\theta):=\theta_0.
\]
Along the common-rotation orbit one has
\[
        g(\theta+\tau(1,\ldots,1))
        =
        g(\theta)+\tau.
\]
Hence the derivative of $g$ in the common-rotation direction is equal to
$1$.  In particular, $g$ is not locally constant on $\mathcal S_m$, and by
the ordinary implicit-function theorem it may be used as one of the two
local coordinates on $\mathcal S_m$.

We may therefore impose the gauge
\begin{equation}\label{eq:mass-persistence-rotation-gauge}
        \theta_0=\theta_0^0.
\end{equation}
Every sufficiently nearby common-rotation orbit meets this gauge exactly
once: given $\theta$, the unique small rotation which brings it to the gauge
is
\[
        \tau=\theta_0^0-\theta_0.
\]
Fixing the coordinate $g=\theta_0^0$ removes exactly one of the two local
parameters.  Thus, after quotienting by common rotation, the nearby
solutions form a smooth one-dimensional local family.

After shrinking the neighborhoods if necessary, all masses remain positive
and all angular positions remain distinct.  Their cyclic order on the
circle is unchanged.  Hence every resulting configuration consists of
distinct points on one circle and is strictly convex.

It remains only to check that the surviving parameter is not removed by the
full similarity quotient.  Let two nearby members of the family, for the
same fixed mass vector, be related by a similarity
\[
        T(z)=a+sAz,
        \qquad
        s>0,
        \qquad
        A\in O(2),
\]
possibly followed by a permutation of particles carrying equal masses.

Every weighted $k$-design with $k\ge1$ has weighted center of mass equal to
the origin.  Since a similarity carries the weighted center of mass to its
image under the same affine map, the two centered configurations can be
similar only if
\[
        a=0.
\]
Moreover, every point of both configurations lies on the unit circle.
Thus
\[
        1=|T(z)|=s|Az|=s
\]
for every point $z$ of the first configuration, and hence
\[
        s=1.
\]
Therefore every remaining similarity identification is induced by an
element of $O(2)$.

The subgroup $SO(2)$ consists precisely of the common rotations which have
already been removed by the gauge
\qref{eq:mass-persistence-rotation-gauge}.  The quotient
\[
        O(2)/SO(2)
\]
has two elements: after common rotations have been factored out, the only
additional orthogonal possibility is one reflection class.  For the fixed
mass vector, the permutations which preserve the masses form a subgroup of
the finite symmetric group on the $2n$ particles, and are therefore finite
as well.

Consequently the remaining equivalence relation on the one-dimensional
gauged solution family has finite equivalence classes.  As in the finite
quotient argument used in the proof of
Theorem~\ref{thm:generic-rotated-design-moduli}, quotienting by these finite
identifications does not change the local dimension.  Hence the full
similarity quotient has local dimension one.  In particular, every
sufficiently small neighborhood contains infinitely many mutually
non-similar weighted configurations.

Finally, every solution of
\[
        \mathcal F(\theta,m)=0
\]
is a weighted spherical $k$-design by
\qref{eq:circle-weighted-fourier}.  Theorem~\ref{thm:even-power-design}
therefore makes every member of the family a central configuration for
\[
        \sigma=2k.
\]
This proves Theorem~\ref{thm:open-mass-persistence}.
\end{proof}

The result should be read together with the general conclusions of
Sections~\ref{sec:spherical-design-mechanism} and
\ref{sec:rotated-design-moduli}.  In arbitrary dimension $d\ge2$, fixed
weighted design blocks already produce generic local families of dimension
\[
        (s-1)\frac{d(d-1)}2
\]
through their independent relative orientations, while
Proposition~\ref{prop:fixed-support-weight-polytope} describes the additional
mass freedom compatible with a fixed support.  The planar theorem proved
above is stronger in a different direction: it shows that the individual
masses themselves may be varied through an open subset of the full positive
mass space, provided the support is allowed to deform simultaneously.

This completes the proofs of the main results of the paper.

\end{document}